\documentclass[leqno]{article}
\usepackage[papersize={176.76mm,250mm},margin=0.75in ]{geometry}
\usepackage{changepage}
\usepackage{fancyhdr} 
\usepackage{etex}
\usepackage{amsmath, bm,amssymb,amsfonts,dsfont,amsthm}
\usepackage{graphicx}
\usepackage[usenames,dvipsnames,svgnames,table]{xcolor}
\usepackage{multicol}
\usepackage{mathrsfs}
\usepackage{mathtools}
\usepackage{amsthm}
\usepackage{enumitem}
\usepackage{lineno}
\definecolor{antiquefuchsia}{rgb}{0.57, 0.36, 0.51}
\definecolor{auburn}{rgb}{0.43, 0.21, 0.1}
\definecolor{darkcerulean}{rgb}{0.03, 0.27, 0.49}
\definecolor{denim}{rgb}{0.08, 0.38, 0.74}
\definecolor{black}{rgb}{0.0, 0.0, 0.0}
\definecolor{sacramentostategreen}{rgb}{0.0, 0.34, 0.25}
\definecolor{phthaloblue}{rgb}{0.0, 0.06, 0.54}
\usepackage[colorlinks=true,linkcolor=black, urlcolor=auburn, filecolor=blue, citecolor=sacramentostategreen,backref=page]{hyperref}

\usepackage{etoolbox}
\makeatletter
\patchcmd{\BR@backref}{\newblock}{\newblock(Cited on page~}{}{}
\patchcmd{\BR@backref}{\par}{)\par}{}{}
\makeatother

\usepackage{relsize}

\usepackage[all]{xy}
\usepackage{tikz-cd}
\usepackage{array}
\usepackage{tensor}
\usepackage[cal=mathpi,calscaled=.94, bb=ams,frakscaled=.97,scr=rsfso]{mathalfa} 
\usepackage[utf8]{inputenc}

\usepackage{pdflscape}
\usepackage{float}

\usepackage{scalerel,stackengine}

\usetikzlibrary{matrix}
\usetikzlibrary{arrows}

\newcommand{\ds}[1]{\ensuremath{\mathds{#1}}}

\DeclareMathOperator{\rr}{\reflectbox{\ensuremath{R}}}

\newcommand{\curly}[1]{\ensuremath{\mathscr{#1}}}
\newcommand{\scr}[1]{\ensuremath{\mathscr{#1}}}
\newcommand{\Os}{\ensuremath{\mathscr{O}}}

\newcommand{\prj}{\ensuremath{\mathrm{Proj}}}
\newcommand{\spc}{\ensuremath{\mathrm{Spec}}}
\newcommand{\spec}{\ensuremath{\mathrm{Spec~}}}
\newcommand{\spf}{\ensuremath{\mathrm{Spf~}}}
\newcommand{\spa}{\ensuremath{\mathrm{Sp~}}}
\newcommand{\rig}{\ensuremath{\mathrm{rig}}}

\newcommand{\pht}{\ensuremath{\hat{\phantom{~}}}}

\newcommand{\nfld}{\mathlarger{\mathlarger{\mathlarger{\mathfrak{o}}}}}

\DeclareMathOperator{\kr}{Ker}
\DeclareMathOperator{\ckr}{Coker}
\DeclareMathOperator{\chr}{Char}

\DeclareMathOperator{\rad}{rad} 
\DeclareMathOperator{\eka}{eka}
\DeclareMathOperator{\Mor}{Mor}

\newcommand{\perfp}{\ensuremath{\ds{P}_K^{n,\mathrm{ad,perf}}}}
\newcommand{\perfq}{\ensuremath{\ds{P}_K^{n,\mathrm{eka}\ds{Q}}}}

\newcommand{\lr}[1]{\left\langle{#1}\right\rangle}

\newcommand{\id}[1]{\ensuremath{{\mathfrak{#1}}}}

\newcommand{\abs}[1]{\ensuremath{{\left\vert{#1}\right\vert}}}

\newcommand{\bs}{\ensuremath{\backslash}}
\newcommand{\xra}[1]{\xrightarrow{#1}}
\newcommand{\ra}{\rightarrow}

\newcommand{\til}[1]{\widetilde{#1}}
\newcommand{\wbr}[1]{{\overline{#1}}}

\usepackage{ccfonts,eucal}
\usepackage[euler-digits,euler-hat-accent]{eulervm}
\usepackage{sseq}

\numberwithin{equation}{section}
\newtheorem{theorem}{Theorem}[section]
\newtheorem*{theorem*}{Theorem}
\newtheorem{lemma}[theorem]{Lemma}
\newtheorem*{lemma*}{Lemma}
\newtheorem{proposition}[theorem]{Proposition}
\newtheorem*{proposition*}{Proposition}
\newtheorem{corollary}[theorem]{Corollary}
\newtheorem*{corollary*}{Corollary}
\newtheorem{mdef}[theorem]{Definition}
\let\olddefinition\mdef
\renewcommand{\mdef}{\olddefinition\normalfont}

\newtheorem{exam}[theorem]{Example}
\let\oldexample\exam
\renewcommand{\exam}{\oldexample\normalfont}

\newtheorem{rem}[theorem]{Remark}
\let\oldremark\rem
\renewcommand{\rem}{\oldremark\normalfont}

\newtheorem{slogan}[theorem]{Slogan}
\let\oldslogan\slogan
\renewcommand{\slogan}{\oldslogan\normalfont}

\newtheorem{notation}[theorem]{Notation}
\let\oldnotation\notation
\renewcommand{\notation}{\oldnotation\normalfont}

\newtheorem{construction}[theorem]{Construction}
\let\oldconstruction\construction
\renewcommand{\construction}{\oldconstruction\normalfont}

\stackMath
\newcommand\rwhat[1]{%
\savestack{\tmpbox}{\stretchto{%
  \scaleto{%
    \scalerel*[\widthof{\ensuremath{#1}}]{\kern-.6pt\bigwedge\kern-.6pt}%
    {\rule[-\textheight/2]{1ex}{\textheight}}
  }{\textheight}%
}{0.5ex}}%
\stackon[1pt]{#1}{\tmpbox}%
}
\usepackage{xfrac}
\newcommand{\myreferences}{/home/harpreet/Dropbox/Biblio/bibThesis}
\usepackage{manfnt}

\begin{document}

\title{Formal Schemes of Rational Degree}
\author{Harpreet Singh Bedi~~~{bedi@gwu.edu}}
\maketitle
 \begin{flushright}{\small In honor of Michel Raynaud $\ldots$ one year later.} \end{flushright}
\begin{abstract}
Non notherian Formal schemes of perfectoid type (for example $\ds{Z}_p[p^{1/p^\infty}]\lr{X^{1/p^\infty} }$ and its multivariate version) with rational degree are constructed and are shown to be admissible. These formal schemes are non Notherian avatar of Tate affinoid algebras. The corresponding notion of topologically finite presentation are constructed and Gabber's Lemma, admissible blow ups (Raynaud's approach) are shown to hold under certain assumptions. A new notion of rings called $\eka^d$ are introduced, which recover most examples of perfectoid affinoid algebras, without resorting to Huber's construction, Witt vectors or Frobenius.
\end{abstract}

\tableofcontents
\section{Introduction}
\subsection{Results}
This paper constructs new Non Notherian formal schemes of rational degree. These formal schemes cover the case of non notherian Perfectoidish formal schemes (inspired by \cite{scholze_1})
of rational degree, but completely avoids Witt Vectors and Frobenius, by adopting a much simpler approach of attaching $d$th power roots (called $\eka^d$ in this paper). The rings are restricted power series of the form $R\lr{X,X^{1/p},X^{1/p^2},\ldots}$ often denoted as $R\lr{X^{1/p^\infty}}$ (in this paper $R\lr{X}_\infty$). The elements are of the form
\begin{equation}
\sum_{i\in{\ds{N}[1/p]}^n}c_iX^i\in R[[X,X^{1/p},X^{1/p^2},\ldots]],\qquad c_i\in R,\qquad\lim_{i\ra\infty}c_i=0.
\end{equation}
The multivariate version is defined similarly as $R\lr{X_1^{1/p^\infty},X_2^{1/p^\infty},\ldots,X_n^{1/p^\infty} }$ (in this paper $R\lr{X_1,X_2,\ldots, X_n}_\infty$).
$R$ itself can be non notherian, for example, $\ds{Z}_p[p^{1/p^\infty}]$.  The degree is often $\ds{Z}[1/p]$, but formal schemes of degree $\ds{Q}$ are also constructed (see section \ref{ratBun}).

In lemma \ref{admissible1} the admissibility of such schemes is proved.
\begin{lemma*}
Let $R$ be an admissible ring with ideal of definition $\id{a}$, then the ring $A=R\lr{X_1,\ldots,X_r}_\infty$ is admissible.
\end{lemma*}
The above lemma makes it possible to do algebraic geometry on perfectoidish schemes in a natural manner. Furthermore, this allows to pass from $R$ to its fraction field using Raynaud's generic fibre approach.

The construction of $\eka^d$ rings is done section \ref{rat1}, and a ton of examples are given. The lemma \ref{admissible1} again holds in the setting of $\eka^d$ rings. In section 7, an analogue of topologically finite type and topologically finite presentation for the rings $R\lr{X_1,\ldots,X_n}_\infty$ are constructed and called topologically finite eka type and topologically finite eka presentation.

The coherence of the rings $R[T^{1/p^\infty}]$ and $R\lr{T}_\infty$ is proved in Proposition \ref{coherent1} and Corollary \ref{coherent2}. 

\begin{proposition*}
Let $A$ be a Noetherian admissible ring with ideal of definition generated by a single element $a$ and $R=A[a^{1/p},a^{1/p^2},\ldots]$ ($R$ is $\eka^p$), then $R[T^{1/p^\infty}]$ is coherent.
\end{proposition*}
\begin{corollary*}
$R\lr{T^{1/p^\infty}}$ is coherent. 
\end{corollary*}

Finally, the following flatness results shown in Lemma \ref{G1} are needed to prove Gabber's Lemma.

\begin{lemma*}
  \begin{enumerate}
    \item  The canonical map $A[T^{1/p^\infty}]\ra A\lr{T^{1/p^\infty}}$ is flat for $A$ Notherian.
    \item The canonical map $K[T^{1/p^\infty}]\ra K\lr{T^{1/p^\infty}}$ is flat for $K$ where $K$ is a field.
    \item The canonical map $R[T^{1/p^\infty}]\ra R\lr{T^{1/p^\infty}}$ is flat for $R$ $\eka^p$.

        \item Let $B=K\lr{\zeta^{1/p^\infty}}$, then the canonical map $B[T]\ra B\lr{T}$ is flat.
    \item Let $B=K\lr{\zeta^{1/p^\infty}}$, then the canonical map $B[T^{1/p^\infty}]\ra B\lr{T^{1/p^\infty}}$ is flat.

  \end{enumerate}

\end{lemma*}
The above mentioned three results form the backbone of the paper. The rest of the results can then be obtained using standard results mentioned in Chapter 7 and 8 of \cite{bosch2014lectures}.

Furthermore, the following Proposition \ref{7.3/8} , Corollary \ref{coro13} and Proposition \ref{7.4/2} are proved, along with other supporting results.
\begin{proposition*}
 Let $A$ be a $R$ algebra of topologically eka type and $M$ a finite $A$ module. Then $M$ is $I$ adically complete and separated.
\end{proposition*}

\begin{corollary*}
Let $A$ be an $R$ algebra that is $I$ adically complete and separated, and let $f_1,\ldots,f_r\in A$ generate the unit ideal. Then the following are equivalent
\begin{enumerate}
\item $A$ is of topologically finite eka presentation (resp. admissible).
\item $A\lr{f_i^{-1}}$ is of topologically finite eka presentation (resp. admissible).
\end{enumerate}
\end{corollary*}

\begin{proposition*}
Let $A$ be an $R$ algebra that is $I$ adically complete and separated, and let $X=\spf A$ be the associated formal $R$ scheme. Then the following are equivalent
\begin{enumerate}
\item $X$ is locally of topologically finite eka presentation (resp. admissible).
\item $A$ is topologically finite eka presentation (resp. admissible) as $R$ algebra.
\end{enumerate}
\end{proposition*}

In the section \ref{coherence1} the coherent properties are discussed, and the following Corollary \ref{8rem4} is shown.

\begin{corollary*}
Let $\curly{M}$ be an $\Os_\id{X}$ module, where $\id{X}$ is a formal $R$ scheme of topologically finite eka presentation, then the following are equivalent
\begin{enumerate}
\item $\curly{M}$ is coherent.
\item $\curly{M}$ is of finite presentation.
\item $\curly{M}{\vert_{X_i}}$ is associated to coherent $\Os_{X_i}$ module, where $(X_i)_{i\in J}$ is a covering of $\id{X}$.
\end{enumerate}

\end{corollary*}

Admissible formal blow ups are discussed in section \ref{blowup1a}. Finally, the following proposition \ref{8.2/7} is shown, which answers the question raised by Peter Scholze about perfectoid spaces and their description as Raynaud Blow Up (at Arizona Winter School 2017) in the setting of $\eka^d$ rings.

\begin{proposition*}
Let $A$ be topologically finite eka presentation and $\id{a}=\lr{f_0,\ldots,f_r}\subset A$ a coherent open ideal. Suppose $\id{X}=\spf A$ is the admissible formal affine $R$ scheme with coherent open ideal $\mathscr{A}=\id{a}^\Delta$ and $\id{X}_\mathscr{A}$ is formal blowing up of $\mathscr{A}$ on $\id{X}$. Then the following hold
\begin{enumerate}
\item The ideal $\scr{A}$ is a line bundle.
\item Let the ideal $\scr{A}$ be generated by $f_i, i=0,\ldots, r$ and $U_i$ be the corresponding locus in $\id{X}_\scr{A}$, then $\{U_i\}$ defines an open affine covering of $\id{X}_\scr{A}$.
\item With $C_i$ as given below, denote $A_i=C_i/(I\mathrm{-torsion})_{C_i}$ then $U_i=\spf A_i$ and the $I$ torsion of $C_i$ is same as the $f_i$ torsion.
\begin{equation}
C_i=A\lr{\frac{f_j}{f_i}}=\frac{A\lr{\xi_j}}{(f_i\xi_j-f_j)}\quad\text{where}\quad j\neq i
\end{equation}

\end{enumerate}

\end{proposition*}

\subsection{Acknowledgement} I am grateful to Matthew Morrow and Peter Scholze for pointing out errors in the earlier draft. These errors are now fixed and any remaining errors are my own.

\section{Linear Topology}
As is customary, the story begins by recalling a few facts from  \cite[pp. 60 \S 7]{PMIHES_1960__4__5_0}(also denoted as EGA0 and EGA1).
\begin{mdef}
 In the linearly topologised ring $A$, we say that an ideal $I$ is an ideal of definition if $I$ is open and if, for all neighborhoods $V$ of $0$, there is an integer $n>0$ such that $I^n\subset V$ (by abuse of language, we say that the sequence ($I^n$) tends to 0). We say that linearly topologised ring $A$ is preadmissible if there exists in $A$ an ideal of definition; we say that A is admissible if it is preadmissible and if it is also separated and complete.
 \end{mdef}

The linear topology allows to transfer neighborhoods of one point to another (linearly), thus only neighborhoods of zero (or some other fixed element) are required to define neighborhood of any other element.
\subsubsection*{Examples}
\begin{enumerate}
\item $\ds{Z}_p$ has infinitely many ideals of definition $p^i\ds{Z}_p, i\in\ds{Z}_{>0}$ and the maximal among them is $p\ds{Z}_p$ (using $p$ adic convergence to zero or algebraically $p$ adic filtration).
\item $\ds{Z}_p[p^{1/p^\infty}]$ contains infinitely many ideals of definition $p^i\ds{Z}_p, i\in \ds{Z}[1/p]_{>0}$. The maximal ideal is $I=\cup_i p^{1/p^i}$ coming from the chain $p\subset p^{1/p}\subset p^{1/p^2}\subset \cdots$. But, the maximal ideal is not an ideal of definition, since $I^n=I$ and hence does not tend to zero.
\item $\ds{Z}_p[\zeta^{1/p^\infty}]$ contains infinitely many ideals of definition $p^i\ds{Z}_p, i\in \ds{Z}_{>0}$ ($p$ adic convergence) forming the neighborhood of zero. Similar to above $I=\cup_i\zeta^{1/p^i}$, where $\zeta$ is the root of unity, is not an ideal of definition (forming the neighborhood of one), since $I^n=I$ and it does not tend to one.
\end{enumerate}
By virtue of corollary below the topologies are independent of the choice of ideal of definition $I$. For semi local rings it is customary to take ideal of definition of a topological ring $A$ as radical of $A$ denoted as $\rad(A)$. In this tract the focus is on linearly topologised rings with finitely generated ideal of definition.

\begin{corollary}[{Corollaire (7.1.8) \cite[pp. 61]{PMIHES_1960__4__5_0}}]
If a preadmissible ring $A$ is such that, for an ideal of definition $I$, the powers $I^n(n>0)$ form a fundamental system of neighborhoods of $0$, it is the same as the powers $I'^n$ of all idéal of définition $I'$ of $A$.
\end{corollary}

\section{Restricted power series}
In [EGA0 \S 7.5] or \cite[pp 212-213]{n1998commutative} the restricted power series are constructed. Let $A$ be a topological ring , linearly topologised, separated and
complete; let ($I_\lambda$) be a fundamental system of neighborhoods of $0$ in $A$ formed by open ideals
(open), so that $A$ is canonically identified with $\varprojlim A/I_\lambda$. For all $\lambda$,
let $B_\lambda= (A/I_\lambda)[T_1, .... T_r]$, where the $T_i$ are the indeterminates ; it is clear that the $B_\lambda$
form a projective systèm of discrete rings. We set $\varprojlim B_\lambda=A\{T_1,\ldots T_n\}$,
and this topological ring  is independent of fundamental systèm
of ideals ($I_\lambda$) considered. More precisely, let $A'$ be a sub ring of the ring of formal series
$A[[T_1,\ldots,T_r]]$ formed by formal
series  $\sum_\alpha c_\alpha T^\alpha$ (with $\alpha=(\alpha_1,\ldots,\alpha_r)\in\ds{N}^r$)
such that $\lim c_\alpha=0$; we
say that these series are the formal series restricted in the $T_i$, with coefficients in $A$. In this tract we re-write $A'=A\lr{T_1,\ldots,T_r}$. There is a topological isomorphism between $A'$ and $\varprojlim_\lambda B_\lambda$ and is shown on [EGA0 \S 7.5, page 70] .

The neighborhoods of $0$ can be defined explicitly in $A'$. For all neighborhoods $V$ of $0$ in $A$, let $V'$ be the set of $x=\sum_\alpha c_\alpha T^\alpha\in A'$ such that $c_\alpha\in V$ for all $\alpha$. The $V'$ form a fundamental system of neighborhoods of $0$ defining on $A'$ a topology of a separated ring.

\begin{rem}\label{7.5.4}
[EGA0 7.5.4].
\begin{enumerate}
\item If $A$ is an admissible, so is $A'=A\lr{T_1,\ldots,T_r}$.
\item Let $A$ be an adic ring, $I$ is an idéal of définition of $A$ such that $I/I^2$ be of finite type
on $A/I$. If we denote $I'=IA',A'$ is then a $I'$ adic ring and $I'/I'^2$ is of finite type on $A'/I'$ . If in addition $A$ is Noetherian, so is $A'$.
\end{enumerate}
\end{rem}

\subsection{Adding $p$th power roots}

\begin{construction}\label{const1a}

Let $R$ be an admissible ring and $\id{a}$ an ideal of definition, consider the following diagram where the first row is just degree $p$ extension ($p$ not necessarily prime), and the second row is its multivariate version.

\begin{equation}{\label{adic2}}
\begin{aligned}
R[X]\ra R[X^{1/p}]\ra R[X^{1/p^2}]\ra &\cdots\ra R[X^{1/p^i}] \ra \cdots \ra \bigcup_{i\geq 0} R[X^{1/p^i}]=\varinjlim_iR[X^{1/p^i}]\\
R[X_1,\ldots,X_n]\ra \cdots\ra R[X_1^{1/p^i},\ldots ,X_n^{1/p^i}]&\ra \cdots \ra  \bigcup_{i\geq 0} R[X_1^{1/p^i},\ldots, X_n^{1/p^i}]=\varinjlim_iR[X_1^{1/p^i},\ldots, X_n^{1/p^i}]
\end{aligned}
\end{equation}
Notice that direct limit $\varinjlim_i$ is simply a union of spaces formed out of injective maps, and the inverse limit $\varprojlim_\lambda$ is taken over the fundamental system of neighborhood formed by $\id{a}$.

In order to ease the flow fix notation as
\begin{equation}
R\lr{X}_\infty=\varprojlim_\lambda\varinjlim_i R[X^{1/p^i}]\quad\text{and}\quad R\lr{X_1,\ldots X_n}_\infty=\varprojlim_\lambda\varinjlim_i R[X_1^{1/p^i},\ldots, X_n^{1/p^i}]\\
\end{equation}

The ring $R\lr{X}_\infty$ is also written as $R\lr{X^{1/p^\infty}}$.

\end{construction}
The elements of $f\in R\lr{X_1,\ldots,X_n}_\infty$ can be explicity described as
\begin{equation}
f=\sum_{i\in\ds{Z}[1/p]_{\geq 0}}a_iT^i\qquad a_i\in R, \quad\abs{a_i}\ra 0 \text{ as }i\ra\infty
\end{equation}
where $T$ denotes the monomial formed by the product of $X_j, j\in\{1,\ldots, n\}$ and $i$ denotes the degree of the monomial. The monomials can be ordered by observing the ordering of rational numbers.
\begin{lemma}\label{admissible1}
Let $R$ be an admissible ring with ideal of definition $\id{a}$, then the ring $A=R\lr{X_1,\ldots,X_r}_\infty$ is admissible.
\end{lemma}
\begin{proof}
The construction above and remark \ref{7.5.4} gives the admissibility. More precisely (following EGA 0,\S 7.5), $A$  is the sub ring of the ring of formal series
$R[[X_1^{1/p^\infty},\ldots,X_r^{1/p^\infty}]]$ formed by formal
series  $\sum_\alpha c_\alpha X^\alpha$ (with $\alpha=(\alpha_1,\ldots,\alpha_r)\in(\ds{N}[1/p])^r$)
such that $\lim c_\alpha=0$; (ordering the monomials by rational degree).
The neighborhoods of $0$ can be defined explicitly in $A$. For all neighborhoods $V$ of $0$ in $R$, let $V'$ be the set of $x=\sum_\alpha c_\alpha X^\alpha\in A$ such that $c_\alpha\in V$ for all $\alpha$. The $V'$ form a fundamental system of neighborhoods of $0$ defining on $A$ a topology of a separated ring.

Another proof can be given by adapting the proof of Proposition 3 from \cite[pp 213]{n1998commutative} by changing the map $\phi_{n_1,\ldots,n_r}^{\id{a}}$ with $(n_1,\ldots,n_r)\in\ds{N}[1/p]$,
\begin{equation}
\phi_{n_1,\ldots,n_r}^{\id{a}}: (A/\id{a})[X_1^{1/p^\infty}, \ldots, X_r^{1/p^\infty}]\ra A/\id{a}
\end{equation}
which maps every polynomial to the coefficient of $X_1^{n_1}\cdots X_r^{n_r}$ in this polynomial. The inverse limit is formed by $\id{a}$ over the neighborhood of zero as in the cited reference.
\end{proof}

\begin{rem}
The above lemma holds for $R=\nfld_K$, where $\nfld_K$ is the valuation ring ( or ring of integers) for the perfectoid field $K$.
\end{rem}
\subsection{Degree as a $\ds{Q}$}
In the spirit of constructions above, the following construction gives power series with rational degree.
\begin{construction}\label{const1b}
Consider the following notations
\begin{equation}
\begin{aligned}
R_1=R[X], R_2=R[X,X^{1/2}],R_3=R[X,X^{1/2}, X^{1/3}],\\
R_4= R[X,X^{1/2}, X^{1/3},X^{1/4}],\ldots, R_i= R[X,X^{1/2}, \ldots,X^{1/i}]
\end{aligned}
\end{equation}

The above rings can be arranged in an increasing order
\begin{equation}
R_1\subset R_2\subset R_3\subset\ldots\subset R_i\subset\ldots \qquad\varinjlim R_i= \bigcup_{i\geq 1} R[X^{1/i}],
\end{equation}
and completed (with respect to ideal of definition) as in the construction \ref{const1a} to get $R\lr{X^{1/i}}, i\in\ds{N}$, in other words $\varprojlim_\lambda\varinjlim_i R_i$. This includes monomial of every rational degree, since the every element of $\ds{Q}_{\geq 0}$ can be written in the form $a/b$ where $a,b\in\ds{N}$.

This construction can be done for the multivariate case in a similar manner. For $n=2$ consider the following
\begin{equation}
\begin{aligned}
&R_1=R[X,Y],\\&R_2=R[X,X^{1/2},Y,Y^{1/2}],\\&R_3=R[X,X^{1/2}, X^{1/3},Y,Y^{1/2}, Y^{1/3}],\\
&R_4= R[X,X^{1/2}, X^{1/3},X^{1/4},Y,Y^{1/2}, Y^{1/3},Y^{1/4}], \\
&\vdots=\qquad\qquad\vdots\\
&R_i= R[X,X^{1/2}, \ldots,X^{1/i},Y,Y^{1/2}, \ldots,Y^{1/i}]\\
&\vdots=\qquad\qquad\vdots.\\
\end{aligned}
\end{equation}
The above can be canonically expanded to the $n$ variable case. Completing the above with respect to ideal of definition gives the series $R\lr{T^i},i\in\ds{Q}$, where $T=\{X_1,\ldots, X_n\}$ and $i$ is the sum of degrees of each $X_j$. The elements can be more formally written as,
\begin{equation}
\sum_ia_iT^i,\qquad i\in\ds{Q}\qquad \abs{a_i}\ra 0,
\end{equation}
and denoted by $\tensor[_{\ds{Q}}]{\rr}{^n}$. The ordering of monomials comes from ordering rationals.
\end{construction}
\begin{rem}
The lemma \ref{admissible1} can be adapted to show admissibility of the ring $\tensor[_{\ds{Q}}]{\rr}{^n}$, by following the exact same reasoning.

\end{rem}

\section{Sheaves}

Let $A$ be an adic ring with ideal of definition $\id{a}$ and is complete and separated. For example, $A=R\lr{X}_\infty$ or $A=R\lr{X_1,\ldots,X_n}_\infty$. It is assumed that the ideal $\id{a}\subset R$ and is finitely generated.

Following [EGA1 10.1 Page 180] or \cite[pp. 158-159]{bosch2014lectures} there is a formal scheme
$\mathfrak{X}=\spf A$  that can be canonically identified with the space $\spc~A/\id{a}$, with $\id{a}$ as an ideal of definition. This gives rise to sheaf, for $f\in A$
\begin{equation}
D(f)\mapsto A\lr{f^{-1}}\quad\text{ and stalk }\quad\curly{O}_x=\varinjlim_{x\in D(f)}A\lr{f^{-1}}
\end{equation}

This sheaf can then be used to compute cohomology of line bundles of the space $\ds{P}^n$ with the same proof as given in \cite{Bedi2018} or \cite{2017arXiv170406820B} for perfectoid spaces. The line bundles $\curly{O}(n)$ now have degree $\ds{Z}[1/p]$ ($p$ not necessarily prime). The cover of the space is given by the hyperplanes $\{X_i=0\}_i$ and setting $f=X_i$ gives the required sheaf for computing with \v{C}ech Cohomology. An extension with rational degree $\ds{Q}$ in stead of $\ds{Z}[1/p]$ is given below.

\subsection{Line Bundle $\curly{O}(d),d\in\ds{Q}$}\label{ratBun}

This brings us to the case of $\curly{O}(d)$ the line bundle with $d\in \ds{Q}$. For example, the global sections of $H^0(\ds{P}^1,\curly{O}_{\ds{P_1}}(2))$ are generated by $x^2,x^{a_1/b_1}y^{a_2/b_2},y^2$ where $(a_1/b_1)+(a_2/b_2)=2$ and $a_i/b_i\in\ds{Q}$, which gives an infinite dimensional space. The affine pieces are series are $\tensor[_{\ds{Q}}]{\rr}{^n}$ glued together in a natural way to get $\perfq$.

\begin{theorem}
\begin{enumerate}
\item $H^0(\perfq,\curly{O}_{\perfq}(m))$ is a free module of infinite rank.
\item $H^n(\perfq,\curly{O}_{\perfq}(-m))$ for $m>n $ is a free module of infinite rank.
\item $H^i(\perfq,\curly{O}_{\perfq})=0$ if $0<i<n$
\end{enumerate}
\end{theorem}

\begin{proof}
The proof from \cite[pages 71-78]{Bedi2018} applies here word for word, by changing $\perfp$ to $\perfq$, which is in turn adapted from \cite{ravi2}.
\end{proof}

The Pic group for $\perfq$ would be $\ds{Q}$, and is computed precisely as the pic group of $\perfp$, which is $\ds{Z}[1/p]$ \cite{Bedi2018}.

\begin{mdef}~\\
\begin{enumerate}
\item Let $A$ be an adic (topological) ring with finitely generated ideal of definition $\id{a}$ and $\mathfrak{X}=\spf A$ with corresponding sheaf of topological rings $\mathcal{O}_\mathfrak{X}$ as constructed above. The locally ringed space $(\mathfrak{X},\mathcal{O}_\mathfrak{X})$ (denoted again by $\mathfrak{X}$) is called the affine formal scheme of $A$.
\item A formal scheme is locally isomorphic to an affine formal scheme, in other words every $x\in\mathfrak{X}$ admits an open neighborhood $U$ where $(U,\mathcal{O}_\mathfrak{X}\vert_U)$ is affine formal scheme.
\item The completed tensor product of two adic rings $A$ and $B$ with ideal of definitions $\id{a},\id{b}$ respectively over a ring $R$ is defined as
\[A\widehat{\otimes}_RB:=\varprojlim A/\id{a}^n\otimes_R B/\id{b}^n\]
$A\hat{\otimes}_R B$ is an adic ring with ideal of definition generated by the image of $\id{a}\otimes_RB+A\otimes_R\id{b}$ if $\id{a},\id{b}$ are finitely generated.
\item
Fiber product of two affine formal schemes $\spf A,\spf B$ over $\spf R$ is given by $\spf(A\widehat{\otimes}_RB)$
\end{enumerate}

\end{mdef}

\subsection{Irreducibility} The irreducible hyperplane $X\in R\lr{X}$ is no longer irreducible in $R\lr{X^{1/p}}$, but its avatar $X^{1/p}$ is irreducible. This is taken as the motivation for definition of codimension one.
\begin{mdef}[codimension 1]
An element $f(X)\in R\lr{X}_\infty$ is defined to be of codimension one if
\[f(X),f(X^{1/p}),\ldots,f(X^{1/p^i}),\ldots\]
 are codimension one in \[R\lr{X},R\lr{X^{1/p}},\cdots,R\lr{X^{1/p^i}},\ldots\] respectively. The above is naturally extended to multivariate case.

 An element $f(X_1,\ldots,X_n)\in R\lr{X_1,\ldots,X_n}_\infty$ is defined to be codimension one if
\[f(X_1,\ldots,X_n),f(X_1^{1/p},\ldots X_n^{1/p}),\ldots,f(X_1^{1/p^i},\ldots,X_1^{1/p^i}),\ldots\]
 are codimension one in \[R\lr{X_1,\ldots X_n},R\lr{X_1^{1/p},\ldots,X_n^{1/p} },\cdots,R\lr{X_1^{1/p^i},\ldots,X_n^{1/p^i}},\ldots\] respectively.

\end{mdef}

\begin{rem}
If the variable $X$ and $X^{1/p^i}$ are interchanged, then $f(X)=g(X)\cdot h(X)$ if and only if $f(X^{1/p^i})=g(X^{1/p^i})\cdot h(X^{1/p^i})$, irreducibility in $R\lr{X_1,\ldots X_n}$ carries over to irreducibility in $R\lr{X_1^{1/p^i},\ldots X_n^{1/p^i}}$ and is thus taken as the avatar of irreducibility in $R\lr{X_1,\ldots X_n}_\infty$. Furthermore, the elements of $R\lr{X}_\infty$ are of the form $f(X,X^{1/p},X^{1/p^2},\ldots)$ and for codim 1, only elements of the form $f(X)$ are being considered.
\end{rem}

The above definition makes it possible to define a cover for a space of the form $R\lr{X_1,\ldots,X_n}_\infty$ and do computations using Weil divisors. For example, the projective space $\ds{P}^n$ can be covered with affine patches defined by $X_i=0$. In fact, the cover of a space will now always be compact since by definition codimension one elements live in notherian ring $R\lr{X_1,\ldots, X_n}$.

\section{Attaching roots to ideal of definition} \label{rat1}

\begin{construction}\label{const1}
In Galois theory roots are routinely attached to extend rings. Let $a\in R$ generate the Ideal of definition. The $d$th root of $a$ can be attached to $R$ as
\begin{equation}
R_1=\dfrac{R[X]}{X^d-a}\qquad\text{ or }\qquad R_1=R[a^{1/d}].
\end{equation}

The above can be repeated inductively to get
\begin{equation}
\begin{aligned}
R_2=\dfrac{R_1[X]}{X^{d^2}-a}\qquad&\text{ or }\qquad R_2=R_1[a^{1/d^2}]=R[a^{1/d},a^{1/d^2}]=R[a^{1/d^2}]\\
R_i=\dfrac{R_{i-1}[X]}{X^{d^i}-a} \qquad&\text{or}\qquad R_i=R_{i-1}[a^{1/d^i}]=R[a^{1/d},a^{1/d^2},\ldots,a^{1/d^i}]=R[a^{1/d^i}].
\end{aligned}
\end{equation}

Thus, according the philosophy adopted in this tract, there is an increasing sequence of rings
\begin{equation}\label{attach2}
R\subset R_1\subset R_2\subset \ldots \subset R_i\subset \ldots \text{ with direct limit } \varinjlim R_i=\bigcup_i R_i=R[a^{1/d}, a^{1/d^2},\ldots]
\end{equation}

\end{construction}

\begin{rem}
It is assumed that the ideal of definition $a$ is finitely generated in this tract. Furthermore, one could add $a^{1/p^i}$ to $R$. This, will not have an impact on topology since the neighborhoods are still considered as $a^n, n\in\ds{N}$  \cite[pp 62, Corollaire 7.1.8]{PMIHES_1960__4__5_0}. \end{rem}

\begin{slogan}
Use Noetherian rings to construct neighborhoods, and then add fractional powers of ideals of definition to make the ring Non-Notherian. The direct limit functor is exact and hence should preserve the properties that are desired.

\end{slogan}

The above leads to following definition.

\begin{mdef}
Let $A$ be an admissible noetherian ring with prescribed maximal ideal of definition generated by a single element $a\in A$, and $a$ has no fractional powers in $A$. Then the ring $A'$ obtained from $A$ by attaching $d$th power roots of $a$ will be called $\eka^d$ ring. (see construction \ref{const1})
\end{mdef}

The word `eka' means one in hindi to represent the principal ideal which has be fractalized into myriad parts. The above definition can be extended in the obvious manner to ideals with multiple generators (attach roots of generators). Moreover, $\eka^1$ recovers the ring $A$ with ideal of definition $a$.

\begin{mdef}
Let $A$ be an admissible ring with prescribed ideal of definition finitely generated by set $\{a_i\}_{i\in J}\in A$, and each $a_i$ have no fractional powers in $A$. Then the ring $A'$ obtained from $A$ by attaching $d$th power roots of $\{a_i\}$  will be called $J\eka^d$ ring.
\end{mdef}

\begin{exam}
Let the pair $(A,I)$ denote the ring and its prescribed ideal of definition.
\begin{enumerate}
\item $(\ds{Z}_p\lr{T},p)$ has $\eka^d$ avatar $\ds{Z}_p[p,p^{1/d}, p^{1/d^2},\ldots]\lr{T}$.
\item $(\ds{Z}_p[[T]],T)$ has $\eka^d$ avatar $\ds{Z}_p[T,T^{1/d},T^{1/d^2},\ldots]$.
\item $(\ds{Z}_p\pht\lr{T},(p,T))$ has $\eka^d$ avatar $\ds{Z}_p[p,T,p^{1/d},T^{1/d}, p^{1/d^2},T^{1/d^2}\ldots]\pht\lr{T}$.
\item $(\ds{Z}_p[T],0)$ has $\eka^d$ avatar $R'$ with lots of nilpotents.
\begin{equation}
R_1=\dfrac{\ds{Z}_p[T][X]}{X^d}\quad\text{ and recursively }\quad R_i=\dfrac{R_{i-1}[X_i]}{X_i^{d^i}} \quad R'=\varinjlim_i R_i.
\end{equation}
A direct way of expressing the above is by writing the $d^i$th roots of zero in terms
of indeterminates as $X_1^d=0,X_2^{d^2}=0,\ldots, X_i^{d^i}=0$ and hence the ring obtained is $S$ below.
\begin{equation}
R=\frac{\ds{Z}_p[X_1,X_2,\ldots][T]}{X_1^d,X_2^{d^2},\ldots}\text{ and } S=\frac{R}{X_1=X_2^d,X_2=X_3^d,\ldots}.
\end{equation}
Notice, that denominator in $S$ is essentially an avatar of a inverse projective limit of frobenius if $d=p$.
\item $(\ds{F}_p,0)$ has $\eka^d$ avatar $\ds{F}_p[p^{1/d},p^{1/d^2},\ldots]$ (also denoted as $\ds{F}_p[p^{1/d^\infty}]$) with lots of nilpotents as in example above.
\end{enumerate}
\end{exam}
\subsubsection{Attaching all rational roots}\label{rat2}

\begin{construction}\label{const2}
In the spirit of construction \ref{const1b}, the rational roots can be attached for the $a\in R$ which generates the ideal of definition.
\begin{equation}
R_2=\dfrac{R[X]}{X^2-a}\qquad\text{ or }\qquad R_2=R[a^{1/2}].
\end{equation}

The above can be repeated inductively to get
\begin{equation}
\begin{aligned}
R_3=\dfrac{R_1[X]}{X^3-a}\qquad&\text{ or }\qquad R_3=R[a^{1/2},a^{1/3}]\\
R_i=\dfrac{R_{i-1}[X]}{X^i-a} \qquad&\text{or}\qquad R_i=R_{i-1}[a^{1/i}]=R[a,a^{1/2},a^{1/3},\ldots,a^{1/i},\ldots].
\end{aligned}
\end{equation}

Thus, according the philosophy adopted in this tract, there is an increasing sequence of rings
\begin{equation}\label{attach2}
R\subset  R_2\subset \ldots \subset R_i\subset \ldots \text{ with direct limit } \varinjlim R_i=\bigcup_i R_i=R[a,a^{1/2},a^{1/3},\ldots,a^{1/i},\ldots].
\end{equation}

\end{construction}

The above constructions naturally lead to the following analogue of Lemma \ref{admissible1}.
\begin{lemma}\label{ekaAdmissible1}
Let $R$ be an admissible $\eka^d$ ring with ideal of definition $\id{a}$, then the ring $A=R\lr{X_1,\ldots,X_r}_\infty$ is admissible.
\end{lemma}
\begin{proof}
  Same proof as that of Lemma \ref{admissible1}.
\end{proof}

\subsubsection{Bijection}
Notice the following bijection
\begin{equation}
\begin{aligned}
\text{Completed Ring with respect to}&\longleftrightarrow \text{Completed Eka ring with respect to}\\
\text{a principal ideal of definition}\qquad&\qquad \text{same ideal of definition}\\
(A,I)&\longleftrightarrow(A',I)\text{ where }A' \text{ is } \eka^p A
\end{aligned}
\end{equation}
This suggests we could consider two categories (i) $\mathcal{D}$ whose objects are Completed Rings with a principal ideal of Definition and (ii) $\mathcal{D}{\eka^p}$ whose objects are Completed eka Rings with same ideal of Definition, but now come attached with all the $p$th power roots. The objects of two categories are in one to one correspondence by construction, for example,
\begin{equation}
\begin{aligned}
\ds{Z}_p&\leftrightarrow\ds{Z}_p[p^{1/p^\infty}], \\
\ds{Q}_p&\leftrightarrow\ds{Q}_p(p^{1/p^\infty}), \\
\ds{F}_p[[T]]&\leftrightarrow\ds{F}_p[[T^{1/p^\infty}]], \\
\ds{F}_p((T))&\leftrightarrow\ds{F}_p((T^{1/p^\infty})). \\
\end{aligned}
\end{equation}

But additional roots bring additional morphisms in $\mathcal{D}{\eka^p}$, making equivalence of categories a non obvious problem.

Furthermore, in the spirit of Raynaud's generic fiber \cite{ray1a}, the above correspondence can be carried over to field of fractions (since, localization in an exact functor).

\subsection{Examples}
Examples of eka rings are given in \ref{e1} and \ref{e3}. Note that ring $\ds{Z}_p\lr{T,T^{1/p},T^{1/p^2},\ldots}$ in \ref{e2} is not an eka since, it does not have $p$th power roots of $p$. Similarly, $\ds{Z}_p\lr{T^r}$ (in  \ref{e4}) is an a non example, but adding $p$th power roots $\ds{Z}_p[p^{1/p^\infty}]\lr{T^r}$ makes it eka.
\begin{enumerate}
\item\label{e1}  A basic example of such a case is $\ds{Z}_p[p^{1/p^\infty}]$, attaching $p$th power roots of $p$ to $p$adic integers. Given any ring $S$ and $t\notin S$, a topological ring of the form $S[[t]]$ can be formed (with $t$-adic valuation) so that neighborhood of zero is given by $t^n, n\in\ds{N}$. Furthermore all $p$th roots of $t$ can be added (which gives $S[[t^{p^{1/\infty}}]]$), to get infinitely many ideals of definition.
\item\label{e2} Ring $R=\ds{Z}[T,T^{1/p},T^{1/p^2},\ldots]$ completed with respect to ideal $pR$ will give the ring $\ds{Z}_p\lr{T,T^{1/p},T^{1/p^2},\ldots}$.
\item Ring $R=\ds{Z}[T,T^{1/d},T^{1/d^2},\ldots]$ completed with respect to ideal $pR$ will give the ring $\ds{Z}_p\lr{T,T^{1/d},T^{1/d^2},\ldots}$ where $d$ is not prime.
\item \label{e4} Ring $R=\ds{Z}[T^r]$ where $r\in\ds{Q}$ completed with respect to ideal $pR$ will give the ring $\ds{Z}_p\lr{T^r}$. Here $T^r$ represents the entire set of $\{T^{r_1},T^{r_2},\ldots\}$  with $r_i\in\ds{Q}$ and the series ordered by the ordering of rational numbers.

\item\label{e3} $R=\ds{Z}_p[p^{1/p^\infty}][T,T^{1/p},T^{1/p^2},\ldots]$ completed with respect to ideal $pR$ will give the ring $\ds{Z}_p[p^{1/p^\infty}]\pht\lr{T,T^{1/p},T^{1/p^2},\ldots}$.
\item $R=\ds{Z}_p[\zeta^{1/p^\infty}][T,T^{1/p},T^{1/p^2},\ldots]$ completed with respect to ideal $pR$ will give the ring $\ds{Z}_p[\zeta^{1/p^\infty}]\pht\lr{T,T^{1/p},T^{1/p^2},\ldots}$. The $\zeta^{1/p^i}$ represents roots of unity.
\item $R=\ds{Z}[\zeta^{1/p^\infty}][T,T^{1/p},T^{1/p^2},\ldots]$ completed with respect to ideal $pR$ will give the ring $\ds{Z}_p[\zeta^{1/p^\infty}]\pht\lr{T,T^{1/p},T^{1/p^2},\ldots}$. The $\zeta^{1/p^i}$ represents roots of unity.
\end{enumerate}

\subsection{Finite Fields $\ds{F}_p$}

Starting with a finite field $\ds{F}_p$ there are two ways in which a fundamental ideal can be attached. First attach an indeterminate $T$ and work with respect to ideal $T$ to form neighborhoods of zero. The second is attach avatars of zero itself, this is the Witt vector approach.

Examples of eka rings are \ref{f2} and \ref{f5}.
\begin{enumerate}
\item Start with ring $\ds{F}_p[T]$ and complete with respect to ideal $T$ to get $\ds{F}_p[[T]]$.
\item \label{f2}Start with ring $\ds{F}_p[T^{1/p^\infty}]$ and complete with respect to ideal $T$ to get $\ds{F}_p[[T^{1/p^\infty}]]$. This can be extended to any rational power roots of $T$.
\item Attach avatars of zero in $\ds{F}_p$, that is $(1,p,p^2,p^3,\ldots)$, which will form neighborhoods of zero. This can be done by the Witt vector, $W(\ds{F}_p)=\ds{Z}_p$.
\item Attach avatars of zero in $\ds{F}_p[T]$, that is $(1,p,p^2,p^3,\ldots)$, which will form neighborhoods of zero. This can be done by the Witt vector, $W(\ds{F}_p[T])$.
\item \label{f5} Attach avatars of zero in $\ds{F}_p[T^{1/p^\infty}]$, that is $(1,p,p^2,p^3,\ldots)$, which will form neighborhoods of zero. This can be done by the Witt vector $W(\ds{F}_p[T^{1/p^\infty}])$.
\end{enumerate}

\subsection{Eka and Separability}
Recall that separability (or Hausdorff) means that $\cap_{i}\id{a}^i=0$ for the ideal of definition $\id{a}$. The intersection comes from the chain
\begin{equation}
\id{a}\supset \id{a}^2\supset\ldots\supset\id{a}^n\supset\ldots\supset\bigcap_{i}\id{a}^i=0
\end{equation}

Adding $d$th power roots extends the chain further
\begin{equation}
\ldots\supset\id{a}^{1/d^2}\supset\id{a}^{1/d}\supset\id{a}\supset \id{a}^2\supset\ldots\supset\id{a}^n\supset\ldots\supset\bigcap_{i}\id{a}^i=0,
\end{equation}
but has no impact on the intersection since bigger sets are being added. Hence, a separable ring still remains separable even after making it $\eka^d$.

\subsection{Eka Valuation Ring}

One can also consider the eka version of Discrete Valuation Rings. Let $\pi$ be the uniformizing parameter of a discrete valuation ring $R$. Then one can construct the $\eka^d$ DVR as $R[\pi^{1/d},\pi^{1/d^2},\ldots,\pi^{1/d^i},\ldots]$. The valuation group is no longer discrete but $\ds{Z}[1/d]$. Furthermore, one could even consider $\eka^\ds{Q}$ version where all rational roots of uniformizer are attached, for example, $R[\pi^i],i\in\ds{Q}_{\geq 0}$.

The residue field $\kappa$ does not change for the $\eka^d$, as shown below.
\begin{equation}
  R/\lr{\pi}=\kappa =R[\pi^{1/d},\pi^{1/d^2},\ldots,\pi^{1/d^i},\ldots]/\cup_i\pi^{1/d^i},\qquad i\in\ds{Z}_{\geq 0}
\end{equation}

Let $f:R\ra R'$ be the morphism of two DVRs with uniformizer of $R$ mapping to uniformizer of $R'$, (local morphism)
Furthermore, let $\kappa'$ (residue field of $R'$) be a finite separable extension of $\kappa$, then $f$ is said to be unramified. Replacing $R$ and $R'$ with their $\eka^d$ avatar, the morphism still remains unramified because the residue fields remain the same.

The Krull dimension of a $R$ is one given by the chain $0\subset \lr{\pi}\subset R$, this chain becomes
\begin{equation}
  0\subset \cup_i\lr{\pi^{1/d^i}}\subset R[\pi^{1/d},\pi^{1/d^2},\ldots,\pi^{1/d^i},\ldots],\qquad i\in\ds{Z}_{\geq 0},
\end{equation}
preserving the Krull dimension. The ideal $\lr{\pi}$ is no longer prime in this ring, since $\pi^{1/d}\cdot\pi^{(d-1)/d}=\pi$ is in the ring but neither $\pi^{1/d}$ or $\pi^{(d-1)/d}$ are in $\lr{\pi}$.

Moreover, if the map $f:R\ra R'$ is flat (in addition to be being unramified), then $f$ is called etale. Again, replacing $R$ and $R'$ with its $\eka^d$ avatar the map $f$ will remain etale. This can be observed by using the ideal criterion for flatness. Injectivity of ideal generated by the uniformizer will be preserved under direct limit. In other words flatness comes from injectivity $\lr{\pi}\otimes_R R'\hookrightarrow R'$, applying the direct limit gives $\varinjlim \lr{\pi}\otimes_{\varinjlim R} \varinjlim R'\hookrightarrow \varinjlim R' $.

\subsection{Eka Local Rings}

Let $R$ be a local ring with maximal ideal $\id{m}$ generated by the set $(a_\lambda)_{\lambda\in\Lambda}$ and the residue field $k=R/\id{m}$. Consider the $J\eka^d$ avatar of $R$ given as $S=\cup_iR[a_\lambda^{1/d^i}], i\in\ds{Z}_{\geq 0}, \lambda\in\Lambda$, then $S$ is a local ring with maximal ideal $\id{m}'=\cup_i(a^{1/d^i}_\lambda)_{\lambda\in\Lambda}$ and the residue field is still $k=R'/\id{m}'$.

\subsection{Eka Rings with Prime Ideals}

Let $R$ be a ring with prime ideal $\id{p}$ generated by the set $(a_\lambda)_{\lambda\in\Lambda}$ and the domain $D=R/\id{p}$. Consider the $J\eka^d$ avatar of $R$ given as $S=\cup_iR[a_\lambda^{1/d^i}], i\in\ds{Z}_{\geq 0}, \lambda\in\Lambda$, then $S$ has a prime ideal $\id{p}'=\cup_i(a^{1/d^i}_\lambda)_{\lambda\in\Lambda}$ and the domain is still $D=R'/\id{p}'$.

Hence, one can naturally do algebraic geometry over non notherian $\eka^d$ rings by emulating the results of Notherian rings.

\subsection{Rees Algebra}
Let $R$ be an admissible ring with an ideal of defintion generated by a single element $a$. Then the rees algebra is given as
\begin{equation}\label{rees1}
  B_aR=R\oplus aR\oplus a^2R \oplus\ldots=\bigoplus_{n\geq 0}a^nR
\end{equation}
The problem with such a blow up is that under reduction $\mod a$ it reduces to $R\mod a$ and all the ideals vanish. In order to preserve the blow up under $\mod a$ mapping it is neccessary to introduce the $\eka^d$ avatar.
\begin{mdef}
  Let $R$ be an $\eka^d$ ring with ideal of definition generated by $a$, then its $\eka^d$ blow up is given as
  \begin{equation}
    R'=\bigoplus_{n\geq 0}a^nR\oplus a^{1/d}R\oplus a^{1/d^2}R \oplus\ldots=  B_aR\oplus\bigoplus_{n\geq 1}a^{1/d^n}R
  \end{equation}
\end{mdef}

This blow up can be transfered from char zero (say $\ds{Z}_p[p^{1/p^\infty}]$) to finite character (say $\ds{F}_p[p^{1/p^\infty}]$).

\subsubsection{Artin Rees}
In this tract the rings are Non Notherian, thus the Artin-Rees Lemma cannot be used directly. Instead a much simpler construction is considered.

Let $R$ be a an $\eka^d$ ring with ideal of defintion $a$ (a single element $a\in R$), and $N\subset M$ be a infinitely (or finitely) generated $R$ modules. Then the corresponding Rees Algebras with respect to ideal $a$ are denoted below.

\begin{equation}
\begin{aligned}
  B_aR&=R\oplus aR\oplus a^2R\oplus\ldots\\
  M'&=M\oplus aM\oplus a^2M\oplus\ldots\\
  N'&=N\oplus aN\oplus a^2N\oplus\ldots\\
\end{aligned}
\end{equation}
If $M$ is infinitely generated by $x_1,x_2,\ldots$ over $R$ they also generate $M'$ over $B_aR$, similary if $N$ is infinitely generated by $y_1,y_2,\ldots$ over $R$ they also generate $N'$ over $B_aR$.

The filtration is $a$ stable, that is $aM_n=M_{n+1}$ for the blow up or filtered graded module $\bigoplus_{n\geq 0}M_n$. Furthermore, $a^nN=a^nM\cap N,n\geq 0$, hence the $a$ adic filtration of $M$ induces $a$ adic filration of the submodule $N$.

\section{Equivalence of categories}\label{tilt13}

A ring homomorphism $\phi:A[X]\ra B$ is called is called an evaluation in $b\in B$ if and only if
\begin{equation}
  \phi(X)=b \text{ and }\phi\circ i=\kappa,
\end{equation}
where $i:A\hookrightarrow A[X]$ and $\kappa:A\ra B$ is a unital ring homomorphism. This ring homomorphism can now be considered for fractional powers, $\phi:A[X^{1/p}]\ra B$ where now $X^{1/p}\mapsto b'$, which would then imply that $X\mapsto b'^p$.

The above example leads us to consider the following direct system of ring homomorphisms with vertical arrows as evaluation homomorphisms $X^{1/p^i}\mapsto b^{1/p^i}, i\in\ds{Z}_{i\geq 0}$.

\begin{equation}\label{dirsys1}
 \begin{tikzpicture}
 []
        \matrix (m) [
            matrix of math nodes,
            row sep=2.5em,
            column sep=1.5em,
                   ]
{ |[name=q1]|A[X] &  |[name=a1]|A[X^{1/p}] & |[name=b1]|\cdots & |[name=c1]|A[X^{1/p^i}]&|[name=u1]|\cdots &|[name=v1]|\varinjlim_iA[X^{1/p^i}]\\
  |[name=q2]|B&  |[name=a2]|B[b^{1/p}] &      |[name=b2]|\cdots & |[name=c2]|B[b^{1/p^i}]&|[name=u2]|\cdots&|[name=v2]|\varinjlim_iB[b^{1/p^i}]\\
        };

        \path[overlay,right hook-latex, font=\scriptsize,>=latex]
(q1) edge (a1)
(a1) edge (b1)
(b1) edge (c1)
(c1) edge (u1)
(u1) edge (v1)

(q2) edge (a2)
(a2) edge (b2)
(b2) edge (c2)
(c2) edge (u2)
(u2) edge (v2)
 ;
 \path[overlay,->, font=\scriptsize,>=latex]
(q1) edge (q2)
(a1) edge (a2)

(c1) edge (c2)

(v1) edge (v2)
;
\end{tikzpicture}
\end{equation}

A special case worth mentioning is the mapping $X\mapsto 0$, then $X^{1/p^i}\mapsto b^{1/p^i}, i\in\ds{Z}_{>0}$ where $b=0$ and $b^{1/p^i}$ represent nilpotents. In particular if we consider the mapping $A[X]\ra B\ra B\mod b$ then the kernel is generated by $b$ and $X\mapsto 0,X\mapsto b$ get mapped to kernel, but it is no longer true in $\varinjlim_iA[X^{1/p^i}]\ra \varinjlim_iB[b^{1/p^i}]$. This fact
 has a topological interpretation, the neighborhoods of zero are formed by the nilpotents.

 Notation

\begin{equation}
\begin{aligned}
  A[X^{1/p^\infty}]:=\varinjlim_iA[X^{1/p^i}]=\bigcup_iA[X^{1/p^i}]&\text{ and }B[b^{1/p^\infty}]:=\varinjlim_iB[b^{1/p^i}]=\bigcup_iB[b^{1/p^i}]\\
  A[X^{1/d^\infty}]:=\varinjlim_iA[X^{1/d^i}]=\bigcup_iA[X^{1/d^i}]&\text{ and }B[b^{1/d^\infty}]:=\varinjlim_iB[b^{1/p^i}]=\bigcup_iB[b^{1/d^i}]
  \end{aligned}
\end{equation}

\subsection{Comparing $\ds{Z}_p[p^{1/d^\infty}][X^{1/p^\infty}] $ to $\ds{F}_p[p^{1/d^\infty}][X^{1/d^\infty}] $}
Inspired by \cite{scholze_1}, we give a new construction relating $\chr 0$ and $\chr p$.
Consider the following mapping
\begin{equation}
\begin{aligned}
\Mor(\ds{Z}_p[X],\ds{Z}_p[X])&\xra{\mod p} \Mor(\ds{F}_p [X],\ds{F}_p [X])\\
X\mapsto 0\text{ and }X\mapsto p&\xra{\mod p} X\mapsto 0\\
\dfrac{\ds{Z}_p[X]}{X}\text{ and }\dfrac{\ds{Z}_p[X]}{X-p}&\xra{\mod p}\dfrac{\ds{F}_p [X]}{X}
\end{aligned}
\end{equation}
which is not one to one. The problem comes from the fact that $X\mapsto 0$ and $X\mapsto p^i,i\in\ds{Z}_{>0}$ will mapto $X\mapsto 0$ in $\Mor(\ds{F}_p [X],\ds{F}_p [X])$, and this arises precisely because the neighborhood of zero generated by the ideal $(p)\in\ds{Z}_p$ has been glued together by the $\mod p$ map, or completely destroyed.

To avoid destruction of neighborhood by $\mod p$ mapping, neighborhood elements that are not destroyed have to be introduced, that is $(p^{1/p}, p^{1/p^2},\ldots)$, this is precisely the $\eka^d$ approach.
\begin{equation}
\begin{aligned}
\Mor(\ds{Z}_p[p^{1/d^\infty}][X^{1/d^\infty}],\ds{Z}_p[p^{1/d^\infty}][X^{1/d^\infty}])&\xra{\mod p} \Mor(\ds{F}_p[p^{1/d^\infty}] [X^{1/d^\infty}],\ds{F}_p[p^{1/d^\infty}] [X^{1/d^\infty}])\\
(X,X^{1/d},X^{1/d^2},\ldots)\mapsto (0,0,0,\ldots) &\xra{\mod p} (X,X^{1/d},X^{1/d^2},\ldots)\mapsto (0,0,0,\ldots)\\
(X,X^{1/d},X^{1/d^2},\ldots)\mapsto (p,p^{1/d},p^{1/d^2},\ldots) &\xra{\mod p} (X,X^{1/p},X^{1/p^2},\ldots)\mapsto (0,p^{1/d},p^{1/d^2},\ldots)\\
\end{aligned}
\end{equation}
  A functor $\mathcal{F}:\scr{A}\ra\scr{B}$ gives equivalence of categories $\scr{A}\equiv\scr{B}$ if it is full, faithful and essentially surjective \cite[pp. 250 Proposition 10.6.2]{hazewinkel2005algebras}.

  Let $\scr{A}$ be the category with one object $\ds{Z}_p[p^{1/d^\infty}][X^{1/d^\infty}]$ and the morphisms are from the object to itself denoted as $\Mor(\ds{Z}_p[p^{1/d^\infty}][X^{1/d^\infty}],\ds{Z}_p[p^{1/d^\infty}][X^{1/d^\infty}])$. These morphisms are given by evaluation maps $X\mapsto T$ and are required to be homomorphisms. Hence, the mapping $X\mapsto T$ means a compatible tuple.
  \begin{equation}
    (X,X^{1/d},X^{1/d^2},\ldots)\mapsto (T,T^{1/d},T^{1/d^2},\ldots),\qquad T^{1/d^i}\in\ds{Z}_p[p^{1/d^\infty}][X^{1/d^\infty}]
  \end{equation}

  Let $\scr{B}$ be the category with one object $\ds{F}_p[p^{1/d^\infty}][X^{1/d^\infty}]$ and the morphisms are from the object to itself denoted as $\Mor(\ds{F}_p[p^{1/d^\infty}][X^{1/d^\infty}],\ds{F}_p[p^{1/d^\infty}][X^{1/d^\infty}])$. These morphisms are given by evaluation maps $X\mapsto t$ and are required to be homomorphisms. Hence, the mapping $X\mapsto t$ means a compatible tuple.
  \begin{equation}
    (X,X^{1/d},X^{1/d^2},\ldots)\mapsto (t,t^{1/d},t^{1/d^2},\ldots),\qquad t^{1/d^i}\in\ds{F}_p[p^{1/d^\infty}][X^{1/d^\infty}]
  \end{equation}
Let $\mathcal{F}:\scr{A}\ra\scr{B}$ be the functor $\mod p$.
\begin{lemma}
  Let $p\nmid d$ and $\mathcal{F}:\scr{A}\ra\scr{B}$ be defined as above, then $\mathcal{F}$ gives an equivalence of categories.
\end{lemma}
\begin{proof} The equivalence follows from the fact that $\mathcal{F}$ satisfies the following.
  \begin{description}
      \item[Essential Surjectivity] The functor $\mod p$ maps the object of $\scr{A}$ to $\scr{B}$.
      \item[$\mathcal{F}$ is surjective] Every morphism in $\scr{B}$ given as $X\mapsto(t,t^{1/d},t^{1/d^2},\ldots)$ which can be lifted in $\scr{A}$ to $X\mapsto(t,t^{1/d},t^{1/d^2},\ldots)$. In other words, lift the tuple as such.
      \item[$\mathcal{F}$ is injective] The injectivity follows from the fact that the kernel of modulo $p$ map is generated by $p^i,i\in\ds{Z}_{>0}$  (and of course zero), and for any $X^{1/d^j}\mapsto p^i,i\in\ds{Z}_{>0}$, there is $X^{1/d^{i+j}}\mapsto p^{i/d^i}$ a non zero element of
$\Mor(\ds{F}_p[p^{1/d^\infty}] [X^{1/d^\infty}],\ds{F}_p[p^{1/d^\infty}] [X^{1/d^\infty}])$. Thus, the only element in the ideal $p\ds{Z}_p$ which maps to zero is zero under the evaluation map.
    \end{description}

\end{proof}

\begin{rem}
  Note that $p\nmid d$ implies that $a+b\neq(a^{1/d}+b^{1/d})^d\mod p$ or re writing $(a+b)^{1/d}\neq a^{1/d}+b^{1/d}\mod p$. Hence $X$ cannot map to $a+X,a\neq 0$ in $\ds{F}_p[p^{1/d^\infty}] [X^{1/d^\infty}]$, since $(a+X)^{1/d^i}$ is not well defined. For the same reason, $X$ cannot map to $a+p^j$. 
\end{rem}

\section{Topologically finite perfectoid Type}\label{top1}
This section studies rings of the form $R\lr{X_1,\ldots X_n}_\infty$, where $R$ is an adic ring with ideal of definition as $I$, consider following two classes \cite[p. 162-169]{bosch2014lectures}:
\begin{enumerate}
\item[(V)] $I$ is finitely generated, hence principal.
\item[(N)] $R$ is Noetherian and it does not have $I$ torsion.
\end{enumerate}
The $\eka^d$ ring is an example of type (V). If $a\in R$ generates the ideal of definition, setting one of the $X_i$ in $R\lr{X_1,\ldots X_n}_\infty$ as $a$ will give a completed $\eka^p$ ring. It can be further assumed that there is no $a$ torsion to prevent any complications.

\begin{rem}
Notice the directed system
\begin{equation}
R\lr{X}\hookrightarrow R\lr{X^{1/p}}\hookrightarrow \cdots\hookrightarrow R\lr{X^{1/p^i}} \hookrightarrow \cdots \hookrightarrow \varinjlim_iR\lr{X^{1/p^i}}= \bigcup_{i\geq 0} R\lr{X^{1/p^i}}.
\end{equation}
There is a strict inclusion $\varinjlim_iR\lr{X^{1/p^i}}\subsetneq R\lr{X}_\infty$ and the rings $R\lr{X^{1/p^i}}, i\in\ds{Z}_{\geq 0}$ are Notherian if $R$ is Notherian.

\end{rem}

\begin{lemma}
$R\lr{X_1,\ldots X_n}_\infty$ is flat over $R$.
\end{lemma}
\begin{proof}
Follows from the ideal criterion of flatness as in Remark 2 on \cite[pp. 163]{bosch2014lectures}.
\end{proof}

\begin{rem}\label{nonnother1}
Consider the non notherian ring $R\lr{T}_\infty$ with non finitely generated ideal $\id{I}_\infty$, then $R\lr{T}_\infty/\id{I}_\infty$ can be considered as an $R\lr{T}_\infty$ module. But, the kernel of the map $R\lr{T}_\infty\ra R\lr{T}_\infty/\id{I}_\infty$ is not finitely generated, hence $R\lr{T}_\infty/\id{I}_\infty$ is not finitely presented as a $R\lr{T}_\infty$ module.

On, the other hand for a Notherian ring $R\lr{T^{1/p^i}}$ the ideal $\id{I}^{1/p^i}$ is finitely generated, and it can be shown that $R\lr{T^{1/p^i}}/\id{I}^{1/p^i}$ is finitely presented.

\end{rem}

\begin{slogan}
The philosophy is to apply $\varinjlim$ to polynomial rings to pass from $R[T]$ to $R[T^{1/p^\infty}]$, this would preserve short exact sequences since $\varinjlim$ is an exact functor, then complete with respect to finitely generated ideal of definition, that is apply the functor $\varprojlim_\lambda$ which would only preserve injectivity.

In other words consider polynomial algebras obtained from power series by reduction modulo the ideal of definition, and then use the direct limit to get polynomial algebras with fractional powers. The techniques of the standard case can then be used.

Recall that a ring $A$ is finitely presented $S$ algebra (or $S\ra A$ is finitely presented) if
\begin{equation}\label{eq1tft}
  A\simeq\dfrac{S[Y_1,\ldots,Y_n]}{f_1,\ldots,f_j}\text{ which implies } \varinjlim A\simeq\varinjlim\dfrac{S[Y_1,\ldots,Y_n]}{f_1,\ldots,f_j}.
\end{equation}
In particular we can talk about finite presentation as a direct limit of the form
\begin{equation}
  \varinjlim A\simeq \dfrac{\left(R[T^{1/p^\infty}]\right)[Y_1,\ldots,Y_n]}{f_1,\ldots,f_j}
\end{equation}
replacing $S$ in \eqref{eq1tft} with non notherian ring $R[T^{1/p^\infty}]$.
\end{slogan}

\begin{proposition}\label{coherent1}
Let $A$ be a Noetherian admissible ring with ideal of definition generated by a single element $a$ and $R=A[a^{1/p},a^{1/p^2},\ldots]$ ($R$ is $\eka^p$), then $R[T^{1/p^\infty}]$ is coherent.
\end{proposition}

\begin{proof}
  Let $(f_1,\ldots,f_n)\in R[T^{1/p^\infty}]$ be a finitely generated ideal, then it lies in the Noetherian ring $B[T^{1/p^{i_1}}]$ where $B=A[a^{1/p^{i_2}}]$ for some $i_1,i_2\in\ds{N}$. Let $i=\max(i_1,i_2)$, set $B_i=(A[a^{1/p^i}])[T^{1/p^i}]$, a Noetherian ring, hence the ideal has a finite presentation of the form
  \begin{equation}
B_i^{m}\ra B_i^{n}\ra(f_1,\ldots,f_n)\ra 0.
  \end{equation}
Applying the direct limit functor and noting that $\varinjlim_iB_i=\cup_iB_i=R[T^{1/p^\infty}]$
\begin{equation}\label{eqcoherent1}
  \begin{aligned}
    \varinjlim_iB_i^{m}\ra \varinjlim_i B_i^{n}\ra &\varinjlim_i(f_1,\ldots,f_n)\ra 0\\
      R[T^{1/p^\infty}]^{m}\ra R[T^{1/p^\infty}]^{n}\ra &(f_1,\ldots,f_n)\ra 0
  \end{aligned}
\end{equation}

\end{proof}

\begin{corollary}\label{coherent2}
$R\lr{T^{1/p^\infty}}$ is coherent. This implies that if $\id{b}\subset R\lr{T^{1/p^\infty}}$ is finitely generated ideal then $R\lr{T^{1/p^\infty}}/\id{b}$ is coherent.
\end{corollary}

\begin{proof}
Any finitely generated of $R\lr{T^{1/p^\infty}}$ can be obtained by completion (or successive $a$ adic approximation) of a finitely generated ideal of $R[T^{1/p^\infty}]$. Thus, apply the completion functor to \eqref{eqcoherent1} (completing with respect to $a$) and observe that inverse limit commutes with direct sum and completion of surjective morphism is surjective (for finite modules) [Tag 0315].
  \begin{equation}\label{coh2eq2}
R\lr{T^{1/p^\infty}}^{m}\ra R\lr{T^{1/p^\infty}}^{n}\ra (f_1,\ldots,f_n)\ra 0
  \end{equation}
Another proof would be to consider the short exact sequence associated to the ideal $M=(f_1,\ldots,f_n)$ and $A=R\lr{T^{1/p^\infty}}$.
\begin{equation}
  0\ra N\ra A^{n}\ra M\ra 0
\end{equation}
and tensor it with $R/a$ to get to polynomial case $R[T^{1/p^\infty}]$, where we know from proposition \ref{coherent1} that the ideal $M/aM$ is finitely presented. Hence, $N/aN$ is a finite $A/aA$ module which implies (standard $a$ adic approximation as given in Theorem 8.4 \cite[pp. 58]{matsumura1989commutative}) that $N$ is a finite $A$ module proving the finite presentation.

\end{proof}

\begin{corollary}
  $K\lr{T^{1/p^\infty}}$ is coherent, where $K$ is the field of fractions of $R$.
\end{corollary}
\begin{proof}
  Apply the right exact functor $-\otimes_R K$ to \eqref{coh2eq2}.
\end{proof}

The following definition is adapted from \cite[pp. 163]{bosch2014lectures} in light of the above remarks.
\begin{mdef}
A topological $R$ algebra $A$, with $R$ of type (V) (or $\eka^d$), is called
\begin{enumerate}
\item of topologically finite  eka type if it is isomorphic to $R\lr{X_1,\ldots,X_n}_\infty/\id{I}_\infty$ with $I$ adic topology, and $\id{I}_\infty$ is an ideal in $R\lr{X_1,\ldots,X_n}_\infty$.
\item of topologically finite eka presentation if in addition to above $\id{I}_\infty$ is finitely generated.
\item admissible if, in addition to the above, $A$ does not have $I$ torsion.
\end{enumerate}
\end{mdef}

\subsection{Artin-Rees}
\begin{rem}\label{ab1}
From \cite[pp 41]{abbes2010elements}
Let $A$ be a ring and $J$ an ideal of $A$, we say that $(A,J)$ verifies
\begin{enumerate}
\item the condition of Artin Rees if $J$ is of finite type and if for all $A$ module $M$ of finite type  and all sub $A$ modules $N$ of $M$, the filtration induced on $N$ by the $J$ preadic filtration on $M$ is bounded; that is to say there exists $n_0$ such that, for all $n\geq n_0$ we say
\[J((J^nM)\cap N)=(J^{n+1}M)\cap N\]
\item \label{krull} the condition of Krull, if for all $A$ module $M$ of finite type and all sub $A$ module $N$ of $M$, the topology $J$ preadic of $N$ is induced by the topology $J$ preadic of $M$.
\end{enumerate}
\end{rem}
The condition \ref{krull} of remark \ref{ab1} is demonstrated below.

\begin{lemma}\label{7.3/7}
Let $R$ be an admissible ring with ideal of definition generated by a single element $\pi$, ($R$ could be non notherian $\eka^d$). Let $A$ be an $R$ algebra of topologically finite eka presentation, $M$ a finite $A$ module, and $N\subset M$. Then the $(\pi)$ adic topology of $M$ restricts to $(\pi)$ adic topology on $N$.

\end{lemma}
\begin{proof}
First note that the generators of $M$ also generate $\pi^i M$, the same is true for $N$. The result follows from the following observation.

\begin{equation}
((\pi^{m+n}M)\cap N)\subset \pi^mN\subset((\pi^mM)\cap N)
\end{equation}

\end{proof}

\begin{proposition}\label{7.3/8}
 Let $A$ be a $R$ algebra of topologically eka type and $M$ a finite $A$ module. Then $M$ is $\pi$ adically complete and separated.
\end{proposition}
\begin{proof}
Follows from Proposition 8 \cite[pp. 165]{bosch2014lectures}, by replacing $R\lr{\zeta}$ with $R\lr{T}_\infty$ and using the lemma \ref{7.3/7}
\end{proof}
\begin{corollary}
Any $R$ algebra of topologically eka type is $I$ adically complete and separated.
\end{corollary}

\subsubsection{Notation}
If $A$ is an $R$ algebra of topologically finite type, set $R_n:=R/I^{n+1}$ and $A_n:=A/I^{n+1}=A\otimes_R R_n$ for $n\in\ds{N}$. Thus, $A$ can be identified with projective limit $\varprojlim_n A_n$. Similar notation is used for $R$ module $M$ with $M_n:=M\otimes_R R_n$

\begin{proposition}\label{7.3/10}
Let $A$ be an $R$ algebra that is $I$ adically complete and separted, then:
\begin{enumerate}
\item $A$ is of topologically finite eka type if and only if $A_0$ is a direct limit of $C_0$ finite type over $R_0$.
\item $A$ is topologically finite eka presentation if and only if $A_n$ is of finite presentation over $R_n$ for all $n\in\ds{N}$.
\end{enumerate}
\end{proposition}

\begin{proof}
\begin{enumerate}
\item The proof is identical to Proposition 10 \cite[pp. 166]{bosch2014lectures}, by applying direct limit to epimorphism $R_0[T]\ra C_0$, gives the epimorphism $\varphi_0:R_0[T^{1/p^\infty}]\ra A_0=\varinjlim C_0$ which in turn gives a continuous onto $R$ algebra homorphism $\varphi:R\lr{T^{1/p^\infty}}\ra A$ as in cited reference.
\item Let $\id{a}=\kr\varphi$ ($\varphi$ as defined above), leading to a SES
\begin{equation}
0\ra \id{a}\ra R\lr{T^{1/p^\infty}}\xra{\varphi}A
\end{equation}
\end{enumerate}
which leads to the following SES (as in the cited reference) where algebras $A_n$ are assumed to be of finite presentation over $R_n$. From the Lemma \ref{7.3/7} we get the exact sequence.
\begin{equation}
0\ra\dfrac{\id{a}}{\id{a}\cap I^nR\lr{T^{1/p^\infty}}}\ra R_n[T^{1/p^\infty}]\ra A_n\ra 0
\end{equation}
As in the cited reference there is a finitely generated ideal $\id{a}'\subset \id{a}$ such that $\id{a}=\id{a}'+I\id{a}$, and limit argument (with respect to $I$) gives $\id{a}=\id{a}'.$
\end{proof}

\begin{proposition}\label{7.3/11}
Let $\varphi:A\ra B$ be a morphism of $R$ algebras of topologically finite eka presentation, and $M$ a finite $B$ module. Then $M$ is flat (resp. faithfully flat) $A$ module if and only if $M_n$ is flat(resp. faithfully flat) $A_n$ module for $n\in\ds{N}$
\end{proposition}
\begin{proof}
See Proposition 11 \cite[pp. 166]{bosch2014lectures} replacing occurence of topologically finite type with topologically finite eka presentation, and use Lemma \ref{7.3/7} (in place of Lemma 7 of cited reference) and Proposition \ref{7.3/8} (in place of Prospostion 8 of cited reference).
\end{proof}

\begin{corollary}\label{7.3/12}
Let $A$ be an $R$ algebra of topologically finite eka presentation, and let $f_1,\ldots, f_r\in A$ generate the unit ideal. Then, all the canonical maps $A\ra A\lr{f_i^{-1}}$ are flat, and $A\ra\prod_{i=1}^r A\lr{f_i^{-1}}$ is faithfully flat.
\end{corollary}
\begin{proof}
See Corollary 12 \cite[pp. 167]{bosch2014lectures} replacing occurence of topologically finite type with topologically finite eka presentation and use Proposition \ref{7.3/11} (in palce of Proposition 8 of the cited reference).
\end{proof}
\begin{corollary}\label{coro13}
Let $A$ be an $R$ algebra that is $I$ adically complete and separated, and let $f_1,\ldots,f_r\in A$ generate the unit ideal. Then the following are equivalent
\begin{enumerate}
\item $A$ is of topologically finite eka presentation (resp. admissible).
\item $A\lr{f_i^{-1}}$ is of topologically finite eka presentation (resp. admissible).
\end{enumerate}
\end{corollary}
\begin{proof}
See Corollary 13 \cite[pp. 167]{bosch2014lectures} replacing occurence of topologically finite type with topologically finite eka presentation. Use Proposition \ref{7.3/10} (in place of Proposition 10 in cited reference) and Corollary \ref{7.3/12} (in place of Corollary 12 in cited reference).

In case $R$ is $\eka$ with ideal of definition generated by a single element $g$, the result follows from the commutative diagram.

\begin{equation}
  \begin{tikzpicture}
 []
        \matrix (m) [
            matrix of math nodes,
            row sep=2.5em,
            column sep=1.5em,
                   ]
{ |[name=q1]|A &  |[name=a1]|~~ & |[name=b1]|\prod_{i=1}^r A\lr{f_i^{-1}} \\
  |[name=q2]|A[g^{-1}] &  |[name=a2]|~~ & |[name=b2]|\prod_{i=1}^r A\lr{f_i^{-1}}[g^{-1}]\\
        };
        \path[overlay,->, font=\scriptsize,>=latex]
 (q1) edge (b1)

  (q2) edge  (b2)

 ;
   \path[overlay,right hook-latex, font=\scriptsize,>=latex]
 (q1) edge  (q2)
  (b1) edge (b2)
   ;
\end{tikzpicture}
\end{equation}

\end{proof}

The following Lemma is again proved in Proposition \ref{G2}.

\begin{lemma}\label{7.3/14}
Let $A$ be an $R$ algebra of topologically finite eka presentation, $B$ an $A$ algebra of finite type, and $M$ a finite $B$ module. Then if $\rwhat{B}$ and $\rwhat{M}$ are the $I$ adic completions of $B$ and $M$, the canonical map
\[M\otimes_b\rwhat{B}\ra \rwhat{M}\]
is an isomorphism.
\end{lemma}

\begin{proof}

This proof closely follows Lemma 14 \cite[pp. 168]{bosch2014lectures}. Start by choosing an exact sequence

\begin{equation}
  0\ra N\ra B^n\xra{p}M \ra 0
\end{equation}
and obtain the first row of the commutative diagram below by applying the $-\otimes_B\rwhat{B}$ (a right exact functor) to the exact sequence. Notice that $h:M\otimes_B\rwhat{B}\ra \rwhat{M}$ is surjective [tag 0315]. It needs to be shown that $h$ is injective to get the isomorphism $h:M\otimes_B\rwhat{B}\ra \rwhat{M}$. This will be obtained by showing $\kr h=0$, by applying snake lemma to \eqref{eq7.3/14} and showing that $\ckr f=0$, or the map $f$ is surjective.

The map $\rwhat{p}$ is surjective since it is composed of surjective maps $\rwhat{B}^n\ra M\otimes_B\rwhat{B}\xra{h}\rwhat{M}$. Let $\wbr{N}$ denote the closure of $N$ in $\rwhat{B}^n$, then $i$ is an inclusion map which can be shown by explicitly finding convergent sequences as in the cited reference.

\begin{equation}\label{eq7.3/14}
\begin{tikzpicture}
 []
        \matrix (m) [
            matrix of math nodes,
            row sep=2.5em,
            column sep=2.5em,
                   ]
{ |[name=w1]|~~ & |[name=q1]|N\otimes_B \rwhat{B} &  |[name=a1]|\rwhat{B}^n & |[name=b1]|M\otimes_B \rwhat{B}& |[name=c1]|0 \\
  |[name=w2]|0 &|[name=q2]|\wbr{N} &  |[name=a2]|\rwhat{B}^n & |[name=b2]|\rwhat{M}& |[name=c2]|0\\
        };
        \path[overlay,->, font=\scriptsize,>=latex]
 (q1) edge (a1)
(a1) edge (b1)
(b1) edge (c1)

(w2) edge (q2)

(a2) edge node[auto] {\(\hat{p}\)} (b2)
(b2) edge (c2)

(q1) edge node[auto] {\(f\)} (q2)

(b1) edge node[auto] {\(h\)}(b2)

 ;

  \path[overlay,right hook-latex, font=\scriptsize,>=latex]
(q2) edge node[auto] {\(i\)}(a2)
  ;
\path[commutative diagrams/.cd, every arrow, every label,font=\scriptsize]
(a1) edge[commutative diagrams/equal](a2)
  ;

\end{tikzpicture}
\end{equation}

The Proposition \ref{7.3/10} yields that $\hat{B}$ is topologically finite eka type, since $A$ is topologically finite eka type and $B$ is finite $A$ algebra. It follows from Proposition \ref{7.3/8} that image of $N\otimes_B \hat{B}$ is closed in $\hat{B}^n$ and thus equals $\wbr{N}$, proving the surjectivity of $f$.

\end{proof}

\subsection{Raynaud Gruson Analogue}
Consider the following SES
\begin{equation}\label{RGeq1}
  0\ra(X_1,X_2,\ldots)\ra R[X_1,X_2,\ldots]\ra \frac{R[X_1,X_2,\ldots]}{(X_1,X_2,\ldots)}\ra 0
\end{equation}

The module ${R[X_1,X_2,\ldots]}/{(X_1,X_2,\ldots)}$ is finitely generated as $R[X_1,X_2,\ldots]$ module but the above SES shows that it is not finitely presented. A finite presentation of a finitely generated $R$ module $M$ translates to finitely generated Kernel of the surjective map $R^n\ra M$ (with $n$ a positive integer).

In fact, the Raynaud Gruson theorem for topologically finite type depends upon finite presentation of finitely generated modules on the Notherian ring $R[T]$. For the case at hand finite presentation of finitely generated modules on the $R[T^{1/p^\infty}]$ is required, but this ring is Non Notherian and can present similar problem as \ref{RGeq1}. Consider the following SES (suggested by Scholze)

\begin{equation}\label{eqsch1}
  0\ra (T,T^{1/p},T^{1/p^2},\ldots)\ra R[T^{1/p^\infty}]\ra \frac{R[T^{1/p^\infty}]}{(T,T^{1/p},T^{1/p^2},\ldots)}\ra 0.
\end{equation}
with infinitely generated kernel. Such, counter examples will be explicitly avoided by making suitable assumptions.

Let $R$ be the $\eka^d$ ring with ideal of definition generated by $t$ and $B$ a topologically finite $R$ algebra and $M$ a finite $B$ module. Then $M/tM$ is a finite $B/tB$ module of finite presentation (since we are now in the world of Notherian rings of the form $R[T]$). This is precisely what does not hold any more for topologically finite eka type as shown in example \eqref{eqsch1}.

\begin{theorem}\label{tfppt}
Let $B$ be an $R$ algebra of topologically finite eka presentation and $M$ a finite $B$ module that is flat over $R$. Furthermore, $M/tM$ is a finite $B/tB$ module of finite presentation. Then $M$ is an $B$ module of finite presentation, that is
\[B^r\ra B^s\ra M\ra 0\]
\end{theorem}

\begin{proof}

Since, $R$ is eka its ideal of definition is generated by a single element say $t$, and $M/tM$ is a finite $B/tB$ module which is flat over $R/tR$. Morevover, $M/tM$ is an $B/tB$ module of finite presentation by assumption.

Consider the following short exact sequence
\begin{equation}\label{7.3/4eq1}
0\ra L\ra B^s\ra M\ra 0
\end{equation}
which will remain exact when tensored with $R/tR$ since $M$ is flat over $R$. As in the reference  \cite[pp. 163]{bosch2014lectures} or Proposition 9.2.1 \cite[pp. 221-222]{fujiwara2018foundations} this gives $L/tL$ is a finite $B/tB$ module (follows from finite presentation of $M/tM$) which implies (standard $t$ adic approximation as given in Theorem 8.4 \cite[pp. 58]{matsumura1989commutative}) that $L$ is a finite $B$ module proving the finite presentation.

\end{proof}

Let us prove the coherence property again.

\begin{lemma}\label{7.3lemma6} \label{7.3/6}
  Let $R$ be an $\eka^d$ ring of the form $A[t,t^{1/p},t^{1/p^2},\ldots]$ where $A$ is Notherian.
\begin{enumerate}
\item The ring $R\lr{T}_\infty$ is coherent.
\item Let $A$ be an $R$ algebra of topologically finite eka presentation then $A$ is coherent.
\end{enumerate}

\end{lemma}

\begin{proof}
\begin{enumerate}
\item It needs to be shown that each finitely generated ideal $\id{a}$ of $R\lr{T}_\infty$ is finitely presented. But any finitely generated ideal is flat over $R$ and $\id{a}/t\id{a}$ will lie in $A[t^{1/p^i}][T^{1/p^i}]$ for $i$ large enough. Since, $A[t^{1/p^i}][T^{1/p^i}]$ is Notherian $\id{a}/t\id{a}$ is of finite presentation. Hence, by Theorem \ref{tfppt} $\id{a}$ has finite presentation.

\item
From previous $\id{a}$ is coherent, thus $R\lr{T}_\infty/\id{a}$ is coherent.

\end{enumerate}
\end{proof}

~

\section{Admissible Formal Schemes}

The recipe for topologically finite type can be followed to get admissible formal schemes. The big difference is that non notherian condition forces topologically finite eka presentation in place of topologically finite eka type as the base model to work with.

\begin{mdef}
Let $X$ be a formal $R$ scheme. $X$ is called locally of topologically finite eka type (resp. locally of topologically finite eka presentation, resp. admissible) if there is an open affine covering $(U_i)_{i\in J}$ of $X$ with $U_i=\spf A_i$ where $A_i$ is an $R$ algebra of topologically finite eka type (resp. of topologically finite eka presentation, resp. admissible).
\end{mdef}

\subsection{Topologically Perfectoid formal scheme}
\begin{proposition}\label{7.4/2}
Let $A$ be an $R$ algebra that is $I$ adically complete and separated, and let $X=\spf A$ be the associated formal $R$ scheme. Then the following are equivalent
\begin{enumerate}
\item $X$ is locally of topologically finite eka presentation (resp. admissible).
\item $A$ is topologically finite eka presentation (resp. admissible) as $R$ algebra.
\end{enumerate}
\end{proposition}

\begin{proof}
Follows from Corollary \ref{coro13}.
\end{proof}

\begin{mdef}
A formal $R$ scheme $X$ is called of topologically finite eka presentation if it is locally of topologically finite eka presentation and quasi compact and quasi separated.
\end{mdef}
\subsection{The functor $\rig$}
In this section we assume the rings to be of locally of topologically finite eka presentation. There is a functor $\rig$ on affine formal $R$ schemes given as below
\begin{equation}
\rig: X=\spf{A}\mapsto X_\rig=\spa(A\otimes_R K)
\end{equation}
where $A$ is topologically finite eka presentation and $\spa(A\otimes_R K)$ is called the category of rigid $K$ spaces (constructed following \cite{ray1a}). It consists of maximal ideals of $A\otimes_R K$ along with functions on it.
\begin{equation}
\begin{aligned}
A\otimes_RK=S^{-1}(R\lr{T}_\infty/\id{I}), &\text{ where }S=R\bs\{0\}\\
R\lr{T}_\infty\subset S^{-1}(R\lr{T}_\infty)\subset K\lr{T}_\infty,& \text{ Canonical Inclusion }
\end{aligned}
\end{equation}

Following, Proposition 6.7.2 of  \cite[pp. 135]{fujiwara2018foundations} we know $K=R[1/\varpi]$ where $\varpi$ is a uniformiser. Since, $\lim c_\nu= 0$ for any series $f=\sum_\nu c_\nu T^\nu\in K\lr{T}_\infty,\nu\in\ds{N}[1/p]$, there is a constant $s\in R\bs\{0\}$ such that $s^{-1}f$ has coefficients in $R$, thus there is an equality $S^{-1}(R\lr{T}_\infty)= K\lr{T}_\infty$, and we get a well defined $X_\rig$.

Moreover, we can apply the above to localization (as in \cite[p. 171]{bosch2014lectures}), since for any $f\in A$ gives
\begin{equation}
\begin{aligned}
A\lr{f^{-1}}\otimes_R K&= \frac{A\lr{T}}{1-fT}\otimes_R K\\
&=\frac{(A\otimes_R K)\lr{T}}{1-fT}=A\otimes_R K\lr{f^{-1}}
\end{aligned}
\end{equation}

It can also be shown that Laurent Domains can be produced from basic open subspace
\begin{equation}
\begin{aligned}
& \begin{tikzpicture}
 []
        \matrix (m) [
            matrix of math nodes,
            row sep=2.5em,
            column sep=1.5em,
                   ]
{ |[name=q1]|A &  |[name=a1]|~~ & |[name=b1]|A\otimes_R K \\
  |[name=q2]|A\lr{f^{-1}} &  |[name=a2]|~~ & |[name=b2]|(A\otimes_R K)\lr{f^{-1}}\\
        };
        \path[overlay,->, font=\scriptsize,>=latex]
 (q1) edge (b1)
 (q1) edge  (q2)
  (q2) edge  (b2)
 (b1) edge (b2)
 ;
\end{tikzpicture}\\
X(f^{-1})&=\spf A\lr{f^{-1}}\subset X=\spf A\\
X_\rig(f^{-1})&=\spa(A\otimes_R K) \lr{f^{-1}}\subset X_\rig=\spa (A\otimes_R K)\\\
\end{aligned}
\end{equation}

The proof on page \cite[p. 171]{bosch2014lectures} can be modified (with a lot work) to give the following theorem \ref{7.4/3}. This is shown in an upcoming paper and is not used in the rest of the article. The main difficulty lies in constructing a well defined functor from the generic fibre $B\otimes_R K\ra A\otimes_R K$ to the corresponding space $\spa(B\otimes_R K)\ra \spa(A\otimes_R K)$.

\begin{theorem}\label{7.4/3}
Let $R$ be a valuation ring of height $1$ with field of fractions $K$ and $A$ an $R$ algebra of topologically finite eka type. The functor $A\mapsto A\otimes_RK$ gives rise to a functor $X\ra X_\rig$ from category of formal $R$ schemes that are locally of topologically finite eka type to category of $K$ spaces.
\end{theorem}

\section{Perfectoid Algebras}
Let $K$ be a perfectoid field and $\wbr{K}$ denote its algebraic closure, and $T$ denote the multi-index $\{X_1,\ldots, X_n\}$. 

The perfectoid algebras will be denoted by $\rr$ (reflected $R$), this can be done by command {
\verb!\DeclareMathOperator{\rr}{\reflectbox{\ensuremath{R}}}! } \\ (or just  {
\verb!\DeclareMathOperator{\rr}{\reflectbox{\text{R}}}! } to avoid complications of font selection) in the preamble and using \verb!\rr! to get $\rr$. This new symbol looks like \emph{pra} in Devnagri (implying perfectoid), and symbols like $\rr_n$ would translate to \emph{pran} meaning life or better $\rr_m$ \emph{prem} for love. 

\begin{mdef}
\begin{enumerate}
\item The $K$ algebra $\rr_n:=K\lr{X_1,\ldots,X_n}_\infty$ of all formal power series
\begin{equation}
\sum_{i\in{\ds{N}[1/p]}^n}c_iT^i\in K[[T]],\qquad c_i\in K,\qquad\lim_{i\ra\infty}c_i=0
\end{equation}
is called perfectoid Tate algebra and the terms can be ordered by degree using the ordering of rational numbers.
\item Let $\id{a}$ be an ideal of $\rr_n$, then algebras of the form $A=\rr_n/\id{a}$ are called perfectoid affinoid $K$ algebras. In other words, there is an epimorphism $\rr_n\ra A$ for some $n\in\ds{N}$.
\end{enumerate}

\end{mdef}
The definition implies that $\rr_n$ converges on the unit disk $\ds{B}^n(\wbr{K})$. 

In this tract the field $K$ perfectoid is replaced by field of fractions of an $\eka^p$ ring and $\wbr{K}$ is consequently meant to say completed $\eka^p$ ring.

\section{Coherent Modules}\label{coherence1}
Let $A$ be a topologically finite eka presentation type $R$ algebra and $M$ an $A$ module. Let $X=\spf(A)$ be the corresponding formal scheme, then we can associate $M^\Delta$ to $M$ defined as 
\begin{equation}
M^\Delta(D_f)=\varprojlim_{n\in\ds{N}}M\otimes_AA_n[f^{-1}]\qquad\text{where}\qquad A_n=A/I^{n+1}
\end{equation}
with $f\in A$ and $I$ is the ideal of definition. 
\begin{proposition}
Let $X=\spf A$ be a formal $R$ scheme of topologically finite eka presentation. If $M$ is an finite $A$ module, the sheaf $M^\Delta$ on the open sets  $D(f),f\in A$ is given by the functor
\begin{equation}
D_f\mapsto M\otimes_A A\lr{f^{-1}}.
\end{equation}

\end{proposition}
\begin{proof}
From corollary \ref{coro13}, $A\lr{f^{-1}}$ is an $R$ algebra of topologically finite eka presentation and proposition \ref{7.3/8} gives $M\otimes_A A\lr{f^{-1}}$ as a finite $A\lr{f^{-1}}$ module is $I$ adically complete and separated. But, $M^{\Delta}(D_f)$ can be considered as the $I$ adic completion of $M\otimes_A A[f^{-1}]$ which is dense in $M\otimes_A A\lr{f^{-1}}$.
\end{proof}

\begin{corollary}\label{8coro2}
Let $X=\spf A$ be a formal $R$ scheme of topologically finite eka presentation, and thus coherent,
\begin{enumerate}
\item There is a fully faithful exact functor $\mathcal{F}:M\ra M^\Delta$ given as
\begin{align*}
\text{Category of coherent $A$ modules}&\ra\text{Category of $\curly{O}_X$ modules}
\end{align*}
\item In the category of $A$ modules the functor $\mathcal{F}$ commutes with kernel, cokernel, image and tensor product. Additionally, the sequence
\begin{equation}
0\ra M'\ra M\ra M''\ra 0
\end{equation}
is exact if and only if the sequence below is exact,
\begin{equation}
0\ra M'^\Delta\ra M^\Delta\ra M''^\Delta\ra 0.
\end{equation}

\end{enumerate}

\end{corollary}

\begin{proof}
Same as Corollary 2 \cite[p. 176-177]{bosch2014lectures} with coherence coming from the lemma \ref{7.3lemma6}. Moreover, the coherence of the module is already assumed (rather than just finiteness condition).
\end{proof}

\begin{mdef}
Let $X$ be a formal $R$ scheme and $\curly{M}$ an $\Os_X$-module.
\begin{enumerate}
\item $\curly{M}$ is called of finite type, if there is an open covering of $X$ given by $(X_i)_{i\in J}$ and corresponding exact sequence
\begin{equation}
\Os_{X}^{s_i}\vert_{X_i}\ra \curly{M}{\vert_{X_i}}\ra 0,\qquad i\in J.
\end{equation}
\item $\curly{M}$ is called of finite presentation, if there is an open covering of $X$ given by $(X_i)_{i\in J}$ and corresponding exact sequence
\begin{equation}
\Os_{X}^{r_i}\vert_{X_i}\ra\Os_{X}^{s_i}\vert_{X_i}\ra \curly{M}{\vert_{X_i}}\ra 0,\qquad i\in J.
\end{equation}
\item $\curly{M}$ is called coherent, if it is of finite type and if the kernel of the morphism below is finite type
\begin{equation}
\Os_{X}^s\vert_U\ra\curly{M}\vert_U\text{ where $U$ is open subscheme of $X$}.
\end{equation}
\end{enumerate}

\end{mdef}

\begin{corollary}\label{8rem4}
Let $\curly{M}$ be an $\Os_\id{X}$ module, where $\id{X}$ is a formal $R$ scheme of topologically finite eka presentation, then the following are equivalent
\begin{enumerate}
\item $\curly{M}$ is coherent.
\item $\curly{M}$ is of finite presentation.
\item $\curly{M}{\vert_{X_i}}$ is associated to coherent $\Os_{X_i}$ module, where $(X_i)_{i\in J}$ is a covering of $\id{X}$.
\end{enumerate}

\end{corollary}
\begin{proof}
The first two are equivalent from the definitions. Assume that $\id{X}=\spf A$ where $A$ is an $R$ algebra of topologically finite eka presentation equipped with an exact sequence
\begin{equation}
(A^r)^\Delta\ra (A^s)^\Delta\ra\curly{M}\ra 0.
\end{equation}
The proof in Remark 8 \cite[p. 177-178]{bosch2014lectures} applies with coherence coming from the lemma \ref{7.3lemma6}.

\end{proof}

\section{Admissible Formal Blowing Up}\label{blowup1a}
The purpose of this section is to prove that formal blowing up of a formal scheme gives a formal scheme again. Since, we are working with non notherian rings Gabber's Lemma would be crucial.
\subsection{Gabber's Lemma}
The first lemma is taken verbatim from \cite[pp. 180]{bosch2014lectures}
\begin{lemma}\label{8.2/1}
Let $M$ be an $A$ module, and $\pi\in A$ a non zero divisor in $A$, then the following are equivalent:
\begin{enumerate}
\item $M$ is flat over $A$.
\item The torsion of $\pi$
\[(\pi\text{ torsion})_M=\{x\in M\text{ such that }\pi^nx=0\text{ for some }n\in\ds{N}\}\]
is trivial in $M$, $M/\pi M$ is flat over $A/\pi A$, and $M\otimes_A A[\pi^{-1}]$ is flat over $A[\pi^{-1}].$
\end{enumerate}
\end{lemma}

The Gabber's result lemma 2 \cite[pp. 181]{bosch2014lectures} is for $R$ algebra of topologically finite type where the rings are Notherian. Here this result is extended to topologically finite eka presentation where the rings are Non Notherian by using results from \cite{fujiwara2018foundations}.

The first result is \cite[pp. 160 Lemma 7.4.9]{fujiwara2018foundations}.
\begin{lemma}\label{7.4.9}
Let $A$ be a ring with ideal $I$ and $M$ an $A$ module and $N$ an $A$ submodule. The topology on $N$ defined by the induced filtration
\begin{equation}
  H^\bullet =\{N\cap I^nM\}_{n\geq 0}
\end{equation}
is $I$ adic if and only if for any $n\geq 0$ there exists $m\geq 0$ such that $N\cap I^mM\subseteq I^nN.$
\end{lemma}
The above lemma is the Krull condition \ref{krull} of remark \ref{ab1} and shown to hold in lemma \ref{7.3/7}.

The second result which requires the conditions of the Lemma \ref{7.4.9} is the following \cite[pp. 161 Proposition 7.4.11]{fujiwara2018foundations}.
\begin{proposition}\label{7.4.11}
Let $A$ be a ring with finitely generated ideal of defintion $I$ and it satisfies the conditions of the Lemma above. Then the exact sequence of finitely generated $A$ modules
\begin{equation}
  N\xra{g}M\xra{f}L
\end{equation}
induces exact sequence

\begin{equation}
    \rwhat{N}\xra{\rwhat{g}}\rwhat{M}\xra{\rwhat{f}}\rwhat{L}
\end{equation}

\end{proposition}

The purpose is to prove the analogue of \cite[pp. 177 Proposition 8.2.18]{fujiwara2018foundations} without using the Notherian outside hypothesis. The analogous result required is the following.

\begin{proposition}\label{G2}
  Let $R$ be an admissible ring with finitely generated ideal of definition $I$, and $M$ be a finite $R$ module with presentation as $0\ra K\ra R^{\oplus n}\ra M\ra 0$. Furthermore the hypothesis of lemma \ref{7.4.9} is satisfied. Then $M\otimes_R\rwhat{R}\ra\rwhat{M}$ is an isomorphism.

\end{proposition}
\begin{proof}

Let $M$ have a presentation as $0\ra K\ra R^{\oplus n}\ra M\ra 0$, then we get the following diagram
\begin{equation}
\begin{tikzpicture}
 []
        \matrix (m) [
            matrix of math nodes,
            row sep=2.5em,
            column sep=1.5em,
                   ]
{ |[name=w1]|~~ & |[name=q1]|K\otimes_R \rwhat{R} &  |[name=a1]|A^{\oplus n}\otimes_R \rwhat{R} & |[name=b1]|M\otimes_R \rwhat{R}& |[name=c1]|0 \\
  |[name=w2]|0 &|[name=q2]|\rwhat{K} &  |[name=a2]|{A^{\oplus n}}^\wedge & |[name=b2]|\rwhat{M}& |[name=c2]|0\\
        };
        \path[overlay,->, font=\scriptsize,>=latex]
 (q1) edge (a1)
(a1) edge (b1)
(b1) edge (c1)

(w2) edge (q2)
(q2) edge (a2)
(a2) edge (b2)
(b2) edge (c2)

(q1) edge (q2)
(a1) edge (a2)
(b1) edge (b2)

 ;
\end{tikzpicture}
\end{equation}

The vertical arrows are surjective since  the map $M\otimes_R\rwhat{R}\ra\rwhat{M}$ is surjective for any finite $R$ module $M$ \cite[tag{0315}]{stacks-project}. The exactness of the first row comes from the right exactness of the tensor product. The exactness of the second row comes from Proposition \ref{7.4.11}. The snake lemma then gives $M\otimes_R\rwhat{R}\ra\rwhat{M}$ is an isomorphism.
\end{proof}
\begin{lemma}\label{G1}
  \begin{enumerate}
    \item  The canonical map $A[T^{1/p^\infty}]\ra A\lr{T^{1/p^\infty}}$ is flat for $A$ Notherian.
    \item The canonical map $K[T^{1/p^\infty}]\ra K\lr{T^{1/p^\infty}}$ is flat for $K$ where $K$ is a field.
    \item The canonical map $R[T^{1/p^\infty}]\ra R\lr{T^{1/p^\infty}}$ is flat for $R$ $\eka^p$.

        \item Let $B=K\lr{\zeta^{1/p^\infty}}$, then the canonical map $B[T]\ra B\lr{T}$ is flat.
    \item Let $B=K\lr{\zeta^{1/p^\infty}}$, then the canonical map $B[T^{1/p^\infty}]\ra B\lr{T^{1/p^\infty}}$ is flat.

  \end{enumerate}

\end{lemma}
\begin{proof} The story boils down to showing a finite presentation of a finitely generated ideal $\id{q}$ of a ring $R$. The proposition \ref{G2} then gives an isomorphism $q\otimes\rwhat{R}=\rwhat{q}$. But, $\rwhat{q}$ is an ideal of $\rwhat{R}$ implying the injectivity of the map $\id{q}\otimes\rwhat{R}\ra\rwhat{R}$ which implies flatness of the map $R\ra\rwhat{R}$. Hence, the lemma follows if there is a finite presentation in each of the cases.
  \begin{enumerate}
    \item Consider a finitely generated ideal $\id{a}$ of $A[T^{1/p^\infty}]=\cup_iA[T^{1/p^i}]$, then $\id{a}\in A[T^{1/p^i}]$ for some $i$ which is notherian. Therefore, $\id{a}$ is finitely presented.
    \item Follows from the previous.
    \item Consider a finitely generated ideal $\id{a}$ of $R[T^{1/p^\infty}]=\cup_iA[a^{1/p^i}][T^{1/p^i}]$, then $\id{a}\in A[a^{1/p^i}][T^{1/p^i}]$ for some $i$ which is notherian. Therefore, $\id{a}$ is finitely presented.

    \item Notice that $B[T]=K\lr{\zeta^{1/p^\infty}}\otimes_KK[T]$. Let $\id{b}$ be a finitely generated ideal of $B[T]$, then $\id{b}=K\lr{\zeta^{1/p^\infty}}\otimes_K\id{b}'$ where $\id{b}'$ is an ideal of $K[T]$ (notherian ring) and hence finitely presented. Apply the functor $K\lr{\zeta^{1/p^\infty}}\otimes_K-$ to the finite presentation of $\id{b}'$ to get a finite presentation of $\id{b}$.

\item Notice that $B[T^{1/p^\infty}]=K\lr{\zeta^{1/p^\infty}}\otimes_KK[T^{1/p^\infty}]$. Let $\id{b}$ be a finitely generated ideal of $B[T^{1/p^\infty}]$, then $\id{b}=K\lr{\zeta^{1/p^\infty}}\otimes_K\id{b}'$ where $\id{b}'$ is an ideal of $K[T^{1/p^i}]$ for some $i$ (notherian ring) and hence finitely presented. Apply the functor $K\lr{\zeta^{1/p^\infty}}\otimes_K-$ to the finite presentation of $\id{b}'$ to get a finite presentation of $\id{b}$.
  \end{enumerate}
\end{proof}
It is time to show Gabber's Lemma (see \cite[Lemma 2 pp. 181]{bosch2014lectures} ), the non notherian part has been taken care of in Lemma \ref{G1}, below is simple base change.
\begin{lemma}\label{8.2/2}
Let $R$ be an adic ring of type (V or $\eka^d$) with ideal of definition generated by $\pi$ and let $A$ be an $R$-algebra of topologically finite eka presentation and $C$ an $A$-algebra
of finite type. Then the $(\pi)$-adic completion $\rwhat{C}$ of $C$ is flat over $C$ .
\end{lemma}
\begin{proof}
  The proof is simple adaptation from the last paragraph of proof Lemma 2 (see \cite[Lemma 2 pp. 181]{bosch2014lectures}).

  Start with a special case $C=A[T]$ (where $A$ is $R\lr{\zeta^{1/p^\infty}}$),  Lemma \ref{8.2/1}  implies that in order to prove $\rwhat{C}$ is flat over $C$, it suffices to show that $\rwhat{C}\otimes_R R[\pi^{-1}]$ is flat over $C\otimes_R R[\pi^{-1}]$ (or $C[\pi^{-1}]$). But, $R[\pi^{-1}]=K$ is a field and thus we need to show
\begin{equation}
\rwhat{C}\otimes_R K=L\lr{T} \text{ is flat over }{C}\otimes_R K=L[T]
  \end{equation}
where $L$ is $K\lr{\zeta^{1/p^\infty}}$ and the flatness follows from Lemma \ref{G1}.

Let $B$ be topologically finite eka type and consider epimorphism $B\ra A$ which can be extended to $\phi:B[Y]\ra C$ since $C$ is of finite type $A$ algebra. Let $\kr\phi:=\id{a}$ which implies that $C=B[Y]/\id{a}$. The flatness of $B\lr{Y}$ over $B[Y]$ (shown above, replace $T$ with $Y$) implies that $B\lr{Y}\otimes_{B[Y]}C=B\lr{Y}/\id{a}B\lr{Y}$ is flat over $C$ (by base change).

  It needs to be shown that $B\lr{Y}/\id{a}B\lr{Y}$ is $\pi$ adic completion of $B[Y]/\id{a}$. Now, consider the mapping $\varphi:B[Y]/\id{a}\ra B\lr{Y}/\id{a}B\lr{Y}$, which can be tensored with $R/(\pi)^n$ to yield isomorphism $\varphi\otimes_RR/(\pi)^n$ of polynomial algebras for any $n$. But,  $ B\lr{Y}/\id{a}B\lr{Y}$ is $\pi$ adically complete and separated by Proposition \ref{7.3/8} and hence it is $\pi$ adic completion of $C=B[Y]/\id{a}$.
\end{proof}

\subsection{Blow Up}
Let $\id{X}$ be a formal $R$ scheme which is locally of topologically finite eka presentation and let $\curly{A}\subset\Os_{\id{X}}$ be an ideal, it will be called open if it contains powers $I^n\Os_{\id{X}}$ where $I$ is the ideal of definition. The proposition \ref{8.1/5} then gives that $\curly{A}$ is associated to a coherent open ideal $\id{a}$ of the underlying $R$ algebra that is topologically finite eka presentation. Following definition 3 \cite[pp. 183]{bosch2014lectures}, along with other details, gives the following.

\begin{mdef}
Let $\mathscr{A}\subset \id{X}$ be a coherent open ideal and $\id{X}$ a formal $R$ scheme which is locally of topologically finite eka presentation. The formal blowing up of $\mathscr{A}$ on $\id{X}$ is the formal $R$ scheme
\begin{equation}
\id{X}_\mathscr{A}=\varinjlim_{n\in\ds{N}}\prj\left( \bigoplus_{d\geq 0}\mathscr{A}^d\otimes_{\Os_{\id{X}}}\frac{\Os_{\id{X}}}{I^n\Os_{\id{X}}}\right)
\end{equation}
along with a canonical projection $\id{X}_\mathscr{A}\ra \id{X}$. A blow up of the form above will be called \emph{admissible formal blow up}.
\end{mdef}
A immediate consequence of the definition is that if $\id{X}=\spf A$ for $\id{X}$ affine, then from proposition \ref{8.1/5} an ideal $\mathscr{A}\subset\Os_{\id{X}}$ is coherent open if and only if it corresponds to open ideal $\id{a}\subset A$. The blow up is now
\begin{equation}
\id{X}_\mathscr{A}=\varinjlim_{n\in\ds{N}}\prj\left(\bigoplus_{d\geq 0}\id{a}^d\otimes_R\frac{R}{I^n} \right).
\end{equation}

The blow up of topologically finite eka presentation might not yield topologically finite eka presentation, but if $R$ does not admit $I$ torsion, the blow up of topologically finite eka presentation yields topologically finite eka presentation.

In light of the above definition, the following results follow from \cite[pp. 185-186]{bosch2014lectures} by changing topologically finite presentation to topologically finite eka presentation.
\begin{proposition}
Admissible formal blowing up commutes with flat base change.
\end{proposition}
\begin{proof}
  Consider the affine case $\id{X}=\spf A,\id{Y}=\spf B$ and $\varphi:\id{X}\ra\id{Y}$ the base change morphism, and $B$ be topologically finite eka type. Let $\scr{A}$ be associated with finitely generated open ideal $\id{a}\subset A$.
\begin{equation}
  \begin{aligned}
    \id{X}_\scr{A}&=\varinjlim_{n\in\ds{N}}\prj\left(\bigoplus_{d=0}^n\id{a}^d\otimes_R\dfrac{R}{I^n} \right)\\
      \id{X}_\scr{A}\times_{\id{X}}\id{Y}&=\varinjlim_{n\in\ds{N}}\prj\left(\bigoplus_{d=0}^n\id{a}^d\otimes_AB\otimes_R\dfrac{R}{I^n} \right)\\
      \id{X}_\scr{A}\times_{\id{X}}\id{Y}&=\varinjlim_{n\in\ds{N}}\prj\left(\bigoplus_{d=0}^n(\id{a}^dB)\otimes_R\dfrac{R}{I^n} \right)\\
  \end{aligned}
\end{equation}
The last equality follows from the fact that if $B$ is flat over $A$ then $\id{a}^i\otimes_AB\ra\id{a}^iB$ is an isomorphism.

\end{proof}

\begin{corollary}
Let $\mathscr{A}\subset \id{X}$ be a coherent open ideal and $\id{X}$ a formal $R$ scheme which is locally of topologically finite eka presentation. Let $U\subset \id{X}$ be a open formal subscheme, the formal blowing up of the coherent ideal $\mathscr{A}|_U\subset\Os_U$ coincides with $\id{X}_\mathscr{A}\times_{\id{X}}U$.
\end{corollary}

Recall that the scheme theoretic blow up for an ideal $\id{a}$ is given as
\begin{equation}
P=\prj\left( \bigoplus_{d\geq 0}\id{a}^d\right) \quad\text{where}\quad \id{a}\subset A
\end{equation}

If $I$ is the ideal of definition of ring $R$, the $I$ adic completion is given as
\begin{equation}
\rwhat{P}=\varinjlim_{n\in\ds{N}}\prj\left( \bigoplus_{d\geq 0}\id{a}^d\otimes_R \frac{R}{I^n}\right)
\end{equation}

The above leads to the following proposition.
\begin{proposition}
Let $\mathscr{A}\subset \id{X}$ be a coherent open ideal, associated with coherent open ideal $\id{a}\subset A$ (i.e. $\mathscr{A}=\id{a}^\Delta$) and $\id{X}$ a formal $R$ scheme which is locally of topologically finite eka presentation. Then the formal blowing up of $\id{X}_\mathscr{A}$ is obtained from the $I$ adic completion of the scheme theoretic blow up of $\id{a}$ on $\spec A$.
\end{proposition}
\begin{proof}
  Start with a scheme theoretic blow up and then complete $P$.
  \begin{equation}
    \begin{aligned}
    P&=\prj\left(\bigoplus_{d=0}^n\id{a}^d \right)\\
        \hat{P}&=\varinjlim_{n\in\ds{N}}\left(P\otimes_R\dfrac{R}{I^n}\right)=\varinjlim_{n\in\ds{N}}\prj\left(\bigoplus_{d=0}^n\id{a}^d\otimes_R\dfrac{R}{I^n} \right)
    \end{aligned}
  \end{equation}
\end{proof}

\subsection{Admissible Blow Up}
In this section the modified Gabber's lemma is required for the results to hold.

\begin{proposition}\label{8.2/7}
Let $A$ be topologically finite eka presentation and $\id{a}=\lr{f_0,\ldots,f_r}\subset A$ a coherent open ideal. Suppose $\id{X}=\spf A$ is the admissible formal affine $R$ scheme with coherent open ideal $\mathscr{A}=\id{a}^\Delta$ and $\id{X}_\mathscr{A}$ is formal blowing up of $\mathscr{A}$ on $\id{X}$. Then the following hold
\begin{enumerate}
\item The ideal $\scr{A}$ is a line bundle.
\item Let the ideal $\scr{A}$ be generated by $f_i, i=0,\ldots, r$ and $U_i$ be the corresponding locus in $\id{X}_\scr{A}$, then $\{U_i\}$ defines an open affine covering of $\id{X}_\scr{A}$.
\item With $C_i$ as given below, denote $A_i=C_i/(I\mathrm{-torsion})_{C_i}$ then $U_i=\spf A_i$ and the $I$ torsion of $C_i$ is same as the $f_i$ torsion.
\begin{equation}
C_i=A\lr{\frac{f_j}{f_i}}=\frac{A\lr{\xi_j}}{(f_i\xi_j-f_j)}\quad\text{where}\quad j\neq i
\end{equation}

\end{enumerate}

\end{proposition}

\begin{proof}~\\
\begin{enumerate}
\item Recall that the scheme theoretic blow up for $X=\spec A$ is given as
$\til{X}=\prj \bigoplus_{d\geq 0}\id{a}^d$ with covering $\bigcup_{i=0}^rD_+(f_i)$ with $D_+(f_i)=\spec S_{(f_i)}$ (homogeneous localization), where $S=\bigoplus_{d\geq 0}\id{a}^d$ and  $\id{a}S_{(f_i)}\subset S_{(f_i)}$ induces invertible ideal $\id{a}\Os_{\til{X}}$.

The above story needs to be carried to formal schemes, in particular flatness has to be exhibited, which would make tensoring an exact functor and everything above would carry to formal schemes. The flatness for topologically finite eka presentation is shown in the modified Gabber's Lemma.

Let $\rwhat{S}_{(f_i)}$ denote the $I$ adic completion of $S_{(f_i)}$, which is flat over $S_{(f_i)}$ by modified Gabber's Lemma and hence the ideal $\id{a}\rwhat{S}_{(f_i)}\subset \rwhat{S}_{(f_i)}$ is invertible.

The formal scheme $\id{X}_\scr{A}$ is covered by $\spf \rwhat{S}_{(f_i)}$, thus $\scr{A}\Os_{\id{X}_\scr{A}}$ is invertible ideal on $\id{X}_\scr{A}$.

\item Since $U_i=\spf \rwhat{S}_{(f_i)}$, it defines an affine cover of $\id{X}$.

\item Start by explicitly describing the localization process and $I$ torsion.
\begin{equation}
B_i=A\left[\frac{f_j}{f_i} \right]= \frac{A\left[{\xi_j}\right]}{(f_i\xi_j-f_j)}\quad\text{where}\quad j\neq i
\end{equation}
The $I$ adic completion of $B_i$ and application of Lemma \ref{7.3/14} gives
\begin{equation}
\begin{aligned}
C_i&=B_i\otimes_{A[\xi_j]}A\lr{\xi_j}\quad\text{where}\quad j\neq i\\
&=A\lr{\frac{f_j}{f_i}}=\frac{A\lr{\xi_j}}{(f_i\xi_j-f_j)}\quad\text{where}\quad j\neq i
\end{aligned}
\end{equation}
By modified Gabber's lemma $C_i$ is flat over $B_i$ implying that
\begin{equation}
\begin{aligned}
(I~\text{torsion})_{C_i}&=(I~\text{torsion})_{B_i}\otimes_{B_i}C_i\\
(f_i~\text{torsion})_{C_i}&=(f_i~\text{torsion})_{B_i}\otimes_{B_i}C_i
\end{aligned}
\end{equation}
If $(I~\text{torsion})_{C_i}=(f_i~\text{torsion})_{C_i}$ Lemma \ref{7.3/14} gives the following with $U_i=\spf A_i$.
\begin{equation}
A_i=\rwhat{S}_{(f_i)}=\frac{A\lr{\dfrac{f_i}{f_j}}}{(I~\text{torsion})}\quad\text{where}\quad j\neq i
\end{equation}
The only thing left to show is that $(I~\text{torsion})_{C_i}=(f_i~\text{torsion})_{C_i}$ which follows from $(I~\text{torsion})_{B_i}=(f_i~\text{torsion})_{B_i}$. Since, $S_{(f_i)}$ does not admit $f_i$ torsion (localization) we get $B_i/(f_i~\text{torsion})\simeq S_{(f_i)}$. The open ideal $\id{a}B_i$ is generated by $f_i$ which gives
\begin{equation}
(f_i~\text{torsion})_{B_i}\subset(I~\text{torsion})_{B_i}.
\end{equation}
But, the formal scheme $\id{X}$ is admissible, hence there is not $I$ torsion in the homogeneous localizations $S_{(f_i)}$ giving the required equality
\begin{equation}
(f_i~\text{torsion})_{B_i}=(I~\text{torsion})_{B_i}.
\end{equation}
\end{enumerate}

\end{proof}
\begin{rem}\label{8.2/8}
The formal blowing up $\id{X}_\scr{A}$ is an admissible formal $R$ scheme, since we always work with finitely generated ideals $\id{a}$. This, condition can be relaxed in the case of Topologically finite type as given in Corollary 8 of \cite[pp. 188]{bosch2014lectures}.
\end{rem}

The proposition \ref{8.2/7} leads to following proposition with the same proofs as given in Proposition 9 and Remark 10 of \cite[pp. 189]{bosch2014lectures}.

\begin{proposition}\label{8.2/9/10}
Let $\id{X}$ be a admissible formal $R$ scheme, and $\scr{A}\subset\Os_\id{X}$ a finitely generated coherent open ideal, with $\id{X}_\scr{A}$ the formal blow up. Then the following hold
\begin{enumerate}
\item Let $\id{Y}$ be a formal scheme equipped with a morphism $\phi:\id{Y}\ra\id{X}$ such that $\scr{A}\Os_{\id{Y}}$ is an invertible ideal in $\Os_{\id{Y}}$, then $\phi$ factors through $\id{X}_\scr{A}$.

\item Let $\scr{B}$ be another finitely generated coherent open ideal in $\Os_\id{X}$ and denote $B=\scr{B}\Os_{\id{X}}$. Then the formal blowing up of the ideal $\scr{AB}$ on $\id{X}$ is isomorphic to blowing up $B$ on $\id{X}_\scr{A}$, given as
\begin{equation}
(\id{X}_\scr{A})_B\ra\id{X}_\scr{A}\ra\id{X}
\end{equation}
\end{enumerate}

\end{proposition}

\begin{proof}
  \begin{enumerate}
    \item The problem can be considered locally. Once we have a proper description of $\id{Y}\ra\id{X}$ we can factorize it through $\id{X}_{\scr{A}}\ra\id{X}$. Let $\id{X}=\spf A, \id{Y}=\spf B$ and $\id{a}=(f_0,\ldots, f_r)\subset A$ associated to $\scr{A}$ such that the ideal $\scr{A}\Os_{\id{Y}}$ is invertible. This means that the ideal $\scr{A}\Os_{\id{Y}}$ is generated by some $f_iB$ and is invertible. Hence, we can consider the unique homomorphism
    \begin{equation}
    A_i:=  A\lr{f_j/f_i}/(f_i\text{-torsion})\ra B, \qquad i\neq j
    \end{equation}
    which extends $A\ra B$ (corresponding to $\id{Y}\ra\id{X}$) mapping fractions $f_j/f_i$ of $A_i$ into $B$.

    The uniqueness follows from the local nature, and factorization has to map locally into $U_i=\spf A_i$. But, $U_i$ is the locus corresponding to the ideal generated by $f_i$.

    \item See Remark 10 of \cite[pp. 189]{bosch2014lectures}.
  \end{enumerate}

\end{proof}

\pagebreak

\bibliographystyle{apalike}

\bibliography{\myreferences}
\bigskip
Assistant Professor of Mathematics and Computer Science\\
Alfred University, 1 Saxon Dr, Alfred, New York 14802

\end{document}